\documentclass[11pt]{article}
\usepackage{smile}

\usepackage{fullpage}
\usepackage{lscape}
\usepackage{bigints}
\usepackage{framed}
\usepackage{mdframed}
\usepackage{enumerate}
\usepackage[inline]{enumitem}
\usepackage[T1]{fontenc}
\usepackage{moresize}
\usepackage{bm}
\usepackage{bbm}
\usepackage{dsfont}
\usepackage{amsmath}
\usepackage{amssymb}
\usepackage{amsthm}
\usepackage{amsfonts}
\usepackage{stmaryrd}
\usepackage{array}
\usepackage{mathrsfs}
\usepackage{mathtools} 
\usepackage{extarrows}
\usepackage{stackrel}
\usepackage{relsize,exscale}
\usepackage{scalerel}
\usepackage[nodisplayskipstretch]{setspace}
\usepackage{color}
\usepackage[usenames,dvipsnames]{xcolor}
\usepackage{cancel}
\usepackage{soul}
\usepackage{xfrac}
\usepackage{siunitx}
\usepackage{graphicx}
\usepackage{float}
\usepackage{rotating}
\usepackage{subcaption}
\usepackage{overpic}
\usepackage[all]{xy}
\usepackage{tikz}
\usetikzlibrary{arrows,matrix,positioning,calc,automata,patterns}
\usepackage{booktabs}
\usepackage{dcolumn}
\usepackage{multirow}
\usepackage{diagbox}
\usepackage{tabularx}
\usepackage{verbatim}
\usepackage{listings}
\usepackage[ruled,vlined]{algorithm2e}
\usepackage{fancyvrb}
\usepackage{hyperref}
\usepackage[round]{natbib}
\usepackage{sectsty}
\usepackage{etoolbox}
\usepackage[capitalize,noabbrev]{cleveref}
\crefname{assumption}{Assumption}{Assumptions}
\newcounter{step}
\AtBeginEnvironment{proof}{\setcounter{step}{0}}
\crefname{step}{Step}{Steps}

\hypersetup{
    bookmarks=true,         
    unicode=false,          
    pdftoolbar=true,        
    pdfmenubar=true,        
    pdffitwindow=false,     
    pdfstartview={FitH},    
    pdftitle={My title},    
    pdfauthor={Author},     
    pdfsubject={Subject},   
    pdfcreator={Creator},   
    pdfproducer={Producer}, 
    pdfkeywords={key1, key2}, 
    pdfnewwindow=true,      
    colorlinks=true,        
    linkcolor=blue,         
    citecolor=blue,         
    filecolor=blue,         
    urlcolor=cyan           
}

\usepackage{stackengine}
\stackMath
\newcommand\tenq[2][1]{%
\def\useanchorwidth{T}%
\ifnum#1>1%
\stackunder[0pt]{\tenq[\numexpr#1-1\relax]{#2}}{\!\scriptscriptstyle\thicksim}%
\else%
\stackunder[1pt]{#2}{\!\scriptstyle\thicksim}%
\fi%
}

\makeatletter
\DeclareRobustCommand\widecheck[1]{{\mathpalette\@widecheck{#1}}}
\def\@widecheck#1#2{%
    \setbox\z@\hbox{\m@th$#1#2$}%
    \setbox\tw@\hbox{\m@th$#1%
       \widehat{%
          \vrule\@width\z@\@height\ht\z@
          \vrule\@height\z@\@width\wd\z@}$}%
    \dp\tw@-\ht\z@
    \@tempdima\ht\z@ \advance\@tempdima2\ht\tw@ \divide\@tempdima\thr@@
    \setbox\tw@\hbox{%
       \raise\@tempdima\hbox{\scalebox{1}[-1]{\lower\@tempdima\box
\tw@}}}%
    {\ooalign{\box\tw@ \cr \box\z@}}}
\makeatother

\newcommand{\Er}{\mathrm{E}}

\newcommand{\Prb}{\mathrm{P}}

\newcommand{\Rb}{\mathbb{R}}

\newcommand{\A}{\mathbf{A}}
\newcommand{\C}{\mathbf{C}}
\newcommand{\D}{\mathbf{D}}
\newcommand{\W}{\mathbf{W}}
\newcommand{\Q}{\mathbf{Q}}

\newcommand{\Step}[2][]{%
  \refstepcounter{step}%
  \paragraph*{Step \thestep: #2.}%
  \if\relax\detokenize{#1}\relax\else\label{#1}\fi
}

\numberwithin{equation}{section}

\newtheorem{theorem}{Theorem}[section]
\newtheorem{lemma}{Lemma}[section]

\newtheorem{corollary}{Corollary}[section]

\newtheorem{innercustomgeneric}{\customgenericname}
\providecommand{\customgenericname}{}
\newcommand{\newcustomtheorem}[2]{%
  \newenvironment{#1}[1]
  {%
   \renewcommand\customgenericname{#2}%
   \renewcommand\theinnercustomgeneric{##1}%
   \innercustomgeneric
  }
  {\endinnercustomgeneric}
}
\newcustomtheorem{customdefinition}{Definition}
\newcustomtheorem{customdefinitions}{Definitions}
\newcustomtheorem{customtheorem}{Theorem}
\newcustomtheorem{customassumption}{Assumption}
\newcustomtheorem{customlemma}{Lemma}
\newcustomtheorem{customexample}{Example}
\theoremstyle{definition}

\newtheorem{example}{Example}[section]

\usepackage{enumitem}
\makeatletter
\newcommand{\mylabel}[2]{#2\def\@currentlabel{#2}\label{#1}}
\makeatother

\graphicspath{{./fig3/}}

\let\tilde\widetilde

\allowdisplaybreaks

\begin{document}

\setlength{\abovedisplayskip}{5pt}
\setlength{\belowdisplayskip}{5pt}
\setlength{\abovedisplayshortskip}{5pt}
\setlength{\belowdisplayshortskip}{5pt}
\hypersetup{colorlinks,breaklinks,urlcolor=blue,linkcolor=blue}

\title{\LARGE Decoupling and randomization for double-indexed permutation statistics}

\author{Mingxuan Zou\thanks{Department of Statistics, University of Chicago, Chicago, IL 60637, USA; e-mail: {\tt mingxuanzou@uchicago.edu}}, ~~Jingfan Xu\thanks{Department of Statistics, University of Washington, Seattle, WA 98195, USA; e-mail: {\tt jfanxu@uw.edu}}, ~~Peng Ding\thanks{Department of Statistics, University of California, Berkeley, CA 94720, USA; e-mail: {\tt pengdingpku@berkeley.edu}}, ~and~Fang Han\thanks{Department of Statistics, University of Washington, Seattle, WA 98195, USA; e-mail: {\tt fanghan@uw.edu}}
}

\date{\today}

\maketitle

\vspace{-1em}

\begin{abstract}
This paper introduces a version of decoupling and randomization to establish concentration inequalities for double-indexed permutation statistics. The results yield, among other applications, a new combinatorial Hanson-Wright inequality and a new combinatorial Bennett inequality. Several illustrative examples from rank-based statistics, graph-based statistics, and causal inference are also provided.
\end{abstract}


\section{Introduction}

Let $N$ be a positive integer, and let $\mathcal{S}_N$ denote the symmetric group consisting of all permutations of $[N] := \{1,2,\ldots,N\}$. Throughout the paper, let $\pi$ be uniformly distributed over $\mathcal{S}_N$. Consider a fourth-order fixed tensor 
\[
\W = [w(i,j,k,\ell)] \in \Rb^{N \times N \times N \times N}.
\]
This paper focuses on the double-indexed permutation statistic (DIPS) of the form 
\begin{align}\label{eq:dips}
Q_w = 
\sum_{i,j \in [N]} w\big(i,j,\pi(i),\pi(j)\big).
\end{align}
It is known that the double-indexed sum in \eqref{eq:dips} admits a decomposition into a simple combinatorial sum and another double-indexed sum induced by a \emph{degenerate} tensor  $\D = [d(i,j,k,\ell)] \in \Rb^{N \times N \times N \times N}$  \citep{zhao1997error, barbour2005permutation}, where degenerate means that all partial averages vanish; see section \ref{sec:main} for formal definition. Specifically, define
\begin{align}\label{eq:degeneratedips}
Q_d 
= \sum_{i,j \in [N]} d\big(i,j,\pi(i),\pi(j)\big).
\end{align}
In this paper, we develop new decoupling and randomization inequalities to simplify the task of bounding the moment-generating function (MGF) of $Q_d$, $\Er\big[\exp(\lambda Q_d)\big]$. 
Based on the bounds on the MGF of $Q_d$, we derive concentration inequalities for $Q_d$, and consequently also for the original statistic $Q_w$.

Decoupling as an approach to handling complex dependent random variables can be traced back at least to \cite{burkholder1983geometric} and \cite{mcconnell1986decoupling}. It has proven particularly successful in studying U-statistics and U-processes \citep{de1992decoupling,arcones1993limit}, as thoroughly summarized in \cite{de2012decoupling}. In the context of our study, a decoupled version of $Q_d$ is
\[
\tilde Q_d = \sum_{i,j\in[N]} d(i,j,\pi(i),\tau(j)),
\]
where $\tau$ is an independent copy of $\pi$. The decoupling inequalities of interest compare $\Er[\exp(\lambda Q_d)]$ with (a variant of) $\Er[\exp(\lambda \tilde Q_d)]$, which is typically much easier to analyze.

Randomization, as a principle for handling complex stochastic systems, has been widely used in statistics and learning theory \citep{van1996weak,vapnik2013nature}. It typically involves comparing moments of the original quantity with those of a ``bootstrapped'' or randomized version. In the context of this paper, a randomization inequality concerns upper-bounding $\Er[\exp(\lambda \tilde Q_d)]$ by the corresponding MGF of
\[
 \sum_{i,j\in[N]} d(i,j,\pi(i),\tau(j))\, G_i G_j',
\]
where $G_1,G_2, \ldots,$ and $G_1',G_2', \ldots$ are independent standard Gaussian random variables, further independent of the system. Such objects are often referred to as ``Gaussian chaos'' in the literature of U-statistics and U-processes \citep{de2012decoupling}.

\subsection{Motivation}

Statistics of the form \eqref{eq:dips} are prevalent in nonparametric statistics and have also received increasing attention in causal inference. In nonparametric statistics, it has been noted that many quantities of statistical interest—including the correlation coefficients of Pearson, Kendall, Spearman, and more recently Chatterjee \citep{chatterjee2021new}, as well as discrepancy measures such as the Mann-Whitney-Wilcoxon statistic—can be expressed in this form and thus benefit from a unified treatment. In survey sampling and causal inference, such statistics have also become increasingly important within the framework of \emph{finite-population inference} \citep{lei2021Regression,han2024introduction}, where they play a role analogous to that of U-statistics in the super-population framework.

In the probability and statistical theory literature, \eqref{eq:dips} may be viewed as natural higher-order extensions of Hoeffding's statistics in the context of sampling without replacement \citep{hoeffding1951combinatorial}. The latter corresponds to the special---and in fact single-indexed---case $d(i,j,k,\ell)=a(i,k)$ for some matrix $\A=[a(i,k)]\in\Rb^{N\times N}$. 

However, unlike Hoeffding's statistics, whose stochastic behavior has been extensively studied (see, e.g., \cite{han2024introduction} for a systematic review), the literature on DIPS is substantially more limited. Notable exceptions include \cite{zhao1997error} and \cite{barbour2005permutation}, who established general Berry-Esseen bounds for $Q_w$, as well as \cite{sambale2022concentration}, who initiated the study of combinatorial concentration inequalities for multilinear forms and obtained useful results for certain special cases of $Q_w$. More general results, however, remain largely unavailable.

From a more technical perspective, although decoupling and randomization techniques have been widely used in the analysis of complex dependent stochastic systems, to the best of our knowledge there has \emph{not} been an extension of these methods to permutation statistics, whose dependence structure is both intricate and ubiquitous.

The purpose of the present work is to make simultaneous contributions to the two research directions outlined above. With respect to the first, we extend the results of \cite{sambale2022concentration} and establish general concentration inequalities for $Q_w$. To achieve this, we also contribute to the second direction by introducing new combinatorial decoupling and randomization techniques that may be useful for broader applications.

The rest of this paper is organized as follows. We begin with four examples that demonstrate the usefulness of the general theory. We then present the main, abstract results in Section \ref{sec:main}. Proofs are given in Section~\ref{sec:proofs}.

\subsection{Generalized combinatorial quadratic forms}

Let $\A=[a_{ij}]\in\Rb^{N\times N}$ be an $N\times N$ matrix, let $\mX=(X_1,\ldots,X_N)^\top$ be a random vector, and consider the quadratic form $Z=\mX^\top \A \mX$. Such random variables were first studied by \cite{hanson1971bound} in the case where $\mX$ is Gaussian, with more structured versions of $\mX$ explored in subsequent works including \cite{rudelson2013hanson}, \cite{adamczak2015note}, and \cite{sambale2022concentration}. 

In particular, \cite{sambale2022concentration} investigated combinatorial quadratic forms. Generalizing their setup, let  $\C=[c_{ij}]\in\Rb^{N\times N}$ be a \emph{doubly centered} matrix, namely,
\[
\sum_{i=1}^N c_{ik} = \sum_{j=1}^N c_{kj} = 0, \quad \text{for all } k\in[N].
\]
Our version of decoupling and randomization enables the following tail bound for the generalized combinatorial quadratic form,
\[
Q_{\sf QF} := \sum_{i,j\in[N]}c_{ij}a_{\pi(i)\pi(j)}.
\]
which coincides with the matrix correlation statistic of \citet{barbour2005permutation}.
\begin{theorem}[A combinatorial Hanson-Wright-type inequality]\label{thm:HansonWright}
Let $N\geq 20$. Assume that $\C$ is doubly centered and $\A =\{a_{ij}\}_{i,j\in[N]}$ has entries in $[-1,1]$. There then exists a universal constant $K>0$ such that the following hold true for any $t>0$.
\begin{enumerate}[itemsep=-.5ex,label=(\roman*)]
\item We have
\[
\Prb\Big(Q_{\sf QF}-\Er[Q_{\sf QF}]\geq t\Big)\leq \exp\Big\{-\frac{Kt^2}{\|\mathbf C\|_{\sf F}^2+(\norm{\C}_{\sf op}+\max_{\sigma\in\mathcal{S}_N}\|\mathbf C\circ \mathbf A^{\sigma}\|_{\sf op})t}\Big\},
\]
where $\norm{\cdot}_{\sf op}$ and $\norm{\cdot}_{\sf F}$ are the matrix spectral and Frobenius norms, respectively, $\A^{\sigma}$ denotes the matrix with entries $a_{\sigma(i)\sigma(j)}$, and ``$\circ$'' represents the Hadamard product.
\item\label{thm:HW2} If we further assume $\A$ to be positive semidefinite, then we have
\[
\Prb\Big(Q_{\sf QF}-\Er[Q_{\sf QF}]\geq t\Big)\leq \exp\left(-\frac{Kt^2}{\|\mathbf C\|_{\sf F}^2+\|\mathbf C\|_{\sf op}t}\right).
\]
\end{enumerate}
\end{theorem}

Note that the above bounds, resembling the classical Hanson-Wright inequality, are dimension-free in the sense that they do not explicitly depend on $N$. Moreover, the transition from subexponential to subgaussian tails depends on the relative magnitudes of $\|\C\|_{\sf F}$ and $\|\C\|_{\sf op}$, mirroring the behavior in both the classical Hanson-Wright and Bernstein-type inequalities. The appearance of the term
\[
\max_{\sigma\in\mathcal{S}_N}\|\C \circ \A^{\sigma}\|_{\sf op}
\]
in the first inequality is somewhat unfortunate, but appears to be unavoidable when working through specific examples. 

The results of \cite{sambale2022concentration} are based on modified log-Sobolev inequalities and apply only to multilinear forms with the special structure $\A=\ma\ma^\top$ for some $\ma\in\Rb^N$. They therefore correspond to a special case of Theorem~\ref{thm:HansonWright}\ref{thm:HW2}. While it may be possible to obtain bounds similar to Theorem~\ref{thm:HansonWright} using analogous log-Sobolev arguments, the tail inequalities presented here are genuinely new. For a concrete application, see Section~\ref{sec:ci}.

Finally, we also obtain a combinatorial Bennett-type inequality that does not require boundedness of the entries of $\A$ or $\C$.

\begin{theorem}[A combinatorial Bennett-type inequality]\label{thm:Bennett}
Assume that $\C$ is doubly centered and has zero diagonal entries. Then there exists a universal constant $K>0$ such that for any $t\ge 0$,
\[
\Prb\Big(Q_{\sf QF}-\Er[Q_{\sf QF}] \ge t\Big)
    \le 
    \exp\Big\{
        -\frac{t}{K \nu}\,
        \log\Big(
            1+\frac{N\nu t}{K\|\C\|_{\sf op}^2 \|\A\|_{\sf F}^2}
        \Big)
    \Big\},
\]
where
\[
\nu := 
\max_{\substack{i,j\in[N]\\ \sigma,\tilde\sigma\in\mathcal S_N}}
\left|
    \sum_{k=1}^{\lfloor N/2 \rfloor} c_{i\sigma(k)}\, a_{j\tilde\sigma(k)}
\right| , 
\]
with $\lfloor \cdot \rfloor$ denoting the floor function.
\end{theorem}
Here $\nu$ measures the largest possible absolute inner product obtainable by selecting half of the coordinates from a row of $\A$ and half of the coordinates from a row of $\C$, and then matching the selected coordinates via an arbitrary bijection. $\lfloor N/2\rfloor$ arises from the decoupling step in the proof, where the index set is split into complementary blocks of sizes $\lceil N/2\rceil$ and $\lfloor N/2\rfloor$. The maximization over $\sigma$ and $\tilde\sigma$ parametrizes all possible ways of selecting and matching these coordinates.

\subsection{Examples in nonparametric statistics}

Main targets in nonparametric statistics include quantifying the stochastic similarity between multiple samples and measuring the strength of dependence within a single sample. In a classical observation, \cite{daniels1944relation} pointed out that many such nonparametric statistics---including Spearman's rho and Kendall's tau---can be formulated in the form \eqref{eq:dips}, and thus fall within the broader family of DIPS.

More specifically, we consider the following four examples.

\begin{example}[Mann-Whitney-Wilcoxon statistic]\label{Mann-Whitney-Wilcoxon Statistic}
Let $X_1,\dots,X_m$ and $Y_1,\dots,Y_n$ be independent samples from continuous distribution functions $F_X$ and $F_Y$, respectively, and set $N=m+n$.
The Mann-Whitney-Wilcoxon statistic $U$ counts the number of pairs with $X_i<Y_j$, namely,
\[
  U = \sum_{i=1}^m \sum_{j=1}^n \ind(X_i < Y_j),
\]
with $\ind(\cdot)$ representing the indicator function.

Pool the data by defining $Z_i := X_i$ for $i\in[m]$ and $Z_{m+j} := Y_j$ for $j\in[n]$. Under $H_0 : F_X = F_Y$, we can show that $U$ has the same distribution as
\[
  Q_{\sf U}
  = \sum_{i,j\in[N]} c_{ij}^{\sf U}\, a_{\pi(i)\pi(j)}^{\sf U},
  \quad \text{with }
  c_{ij}^{\sf U} := \ind(i \le m,\ j > m),
  \quad
  a_{k\ell}^{\sf U} := \ind(Z_k < Z_\ell).
\]

Applying Corollary~\ref{cor:uncentered2} below 
and performing straightforward calculations, we obtain that there exists a universal constant $K>0$ such that, for all $t>0$,
\[
  \Prb\Big(Q_{\sf U} - \Er[Q_{\sf U}] > t\Big)
  \le
  \exp\Big(- \frac{K\, t^2}{\,Nmn + Nt\,}\Big)
  \;+\;
  \exp\Big(- \frac{K\, t^2}{\,\min\{m,n\}^2 + \min\{m,n\} t\,}\Big).
\]
\end{example}

\begin{example}[Daniels’ generalized correlation coefficient]\label{example:daniels}
Assume that $\{(X_i,Y_i)\}_{i\in[N]}$ are independent and identically distributed copies of a representative pair $(X,Y)$ with continuous marginals $F_X$ and $F_Y$.
Assign to each pair $(X_i,X_j)$ a score $c_{ij}^{\sf R}$ and to each $(Y_i,Y_j)$ a score $a_{ij}^{\sf R}$, with $c_{ij}^{\sf R}=-c_{ji}^{\sf R}$ and $a_{ij}^{\sf R}=-a_{ji}^{\sf R}$.  
Let $\C^{\sf R}=[c_{ij}^{\sf R}]$ and $\A^{\sf R}=[a_{ij}^{\sf R}]$ denote the corresponding antisymmetric score matrices.  

\citet{daniels1944relation} introduced the generalized correlation coefficient 
\[
  R := 
  \frac{
    \sum_{i,j=1}^N c_{ij} a_{ij}
  }{
    \Big( \sum_{i,j=1}^N c_{ij}^2 \Big)^{1/2}
    \Big( \sum_{i,j=1}^N a_{ij}^2 \Big)^{1/2}
  }.
\]
Under 
\begin{align}\label{eq:null-independence}
  H_0: X \text{ is independent of } Y,
\end{align}
$R$ has the same distribution as 
\[
  Q_{\sf R} 
  = \frac{\sum_{i,j\in[N]} c_{ij}^{\sf R}\, a_{\pi(i)\pi(j)}^{\sf R}}{{\left(\sum_{i,j=1}^N (c^{R}_{ij})^2\right)^{1/2}\left(\sum_{i,j=1}^N (a^{R}_{ij})^2\right)^{1/2}}}.
\]

\citet{daniels1944relation}'s framework encompasses several classical statistics as special cases, including Pearson’s  correlation, Spearman’s rank correlation $\rho$, and Kendall’s $\tau$, under suitable choices of score functions:
\begin{enumerate}[itemsep=-.5ex,label=(\roman*)]
\item Pearson's correlation coefficient ($\rho_{\sf P}$) corresponds to
\[
  c_{ij}^{\sf R} = X_i-X_j, 
  \qquad 
  a_{ij}^{\sf R} = Y_i-Y_j.
\]
\item Kendall's tau ($\rho_{\sf K}$) corresponds to
\[
  c_{ij}^{\sf R} = \operatorname{sgn}(X_i-X_j), 
  \qquad 
  a_{ij}^{\sf R} = \operatorname{sgn}(Y_i-Y_j),
\]
where $\operatorname{sgn}(x) := \ind(x>0)-\ind(x<0)$ is the sign function.

\item Spearman's rho ($\rho_{\sf S}$) corresponds to
\[
  c_{ij}^{\sf R} = r_i^x-r^x_j, 
  \qquad 
  a_{ij}^{\sf R} = r_i^y-r^y_j,
\]
where $r_i^x$ denotes the ranks of $X_i$ in the samples
$\{X_k\}_{k\in[N]}$ and  $r_i^y$ denotes the ranks of $Y_i$ in $\{Y_k\}_{k\in[N]}$.
\end{enumerate}

Under $H_0$ in \eqref{eq:null-independence}, Corollary \ref{cor:uncentered2} below implies that there exists a universal constant $K>0$ such that, for all $t>0$,
\begin{align*}
  \Prb\Big( \rho_{\sf P} - \Er[\rho_{\sf P}] \ge t \Big)
  &\le 
  \exp\Big(
    -\,\frac{
        K\, t^2
    }{
        N^{-1} 
        + 
        \dfrac{
            \max_i |X_i - \overline X| \, \max_k |Y_k - \overline Y|
        }{
            S_X S_Y
        }\,
        t
    }
  \Big),
    \\
  \Prb\Big( \rho_{\sf K} - \Er[\rho_{\sf K}] \ge t \Big)
  &\le 
  \exp\Big( - K N t^2 / (1 + t) \Big)
  + 
  \exp\Big( - K N^2 t^2 / (1 + N t) \Big),
  \\
  \Prb\Big( \rho_{\sf S} - \Er[\rho_{\sf S}] \ge t \Big)
  &\le 
  \exp\Big( - K N t^2 / (1 + t) \Big),
\end{align*}
where
\[
  \overline X := N^{-1} \sum_{i=1}^N X_i, 
  \qquad 
  \overline Y := N^{-1} \sum_{i=1}^N Y_i,
\]
and
\[
  S_X := \Big( \sum_{i=1}^N (X_i - \overline X)^2 \Big)^{1/2},
  \qquad
  S_Y := \Big( \sum_{i=1}^N (Y_i - \overline Y)^2 \Big)^{1/2}.
\]
\end{example}

\begin{example}[Chatterjee's rank correlation]\label{example:chatterjee}
\citet{chatterjee2021new} has recently introduced a rank correlation as an appealing alternative to the correlation coefficients in Example~\ref{example:daniels}; see, e.g., \cite{dette2013copula}, \cite{han2021extensions}, \cite{chatterjee2024survey}, and \cite{lin2022limit} for more discussions.  
Under \eqref{eq:null-independence}, 
Chatterjee's rank correlation has the same distribution as
\[
  \xi_N
  = 1 - \frac{3}{N^2 - 1} \sum_{i=1}^{N-1} \big| \pi(i+1) - \pi(i) \big|.
\]
Taking $c_{ij}^{\sf\xi} := \ind(i-j=1)$ and $a_{k\ell}^{\sf\xi} := \big| k - \ell \big|$, we can rewrite $\xi_N$ as
\[
  \xi_N
  = 1 - \frac{3}{N^2 - 1}
    \sum_{i,j\in[N]} c_{ij}^{\sf\xi} \, a_{\pi(i)\pi(j)}^{\sf\xi}.
\]
Thus, Chatterjee’s rank correlation admits a representation as a combinatorial quadratic form as well and its stochastic properties under $H_0$ in  \eqref{eq:null-independence} can be analyzed analogously.

More specifically, applying Corollary~\ref{cor:uncentered2}, we obtain the existence of a universal constant $K>0$ such that, for all $t>0$,
\[
  \Prb\big( \xi_N - \Er[\xi_N] \ge t \big)
  \le 
  \exp\Big( - \frac{K N^2 t^2}{1 + Nt} \Big)
  \;+\;
  \exp\Big( - \frac{K N t^2}{1 +\, t} \Big).
\]
\end{example}

\begin{example}[Friedman and Rafsky's graph correlation]
\citet{friedman1983graph} introduced a graph-based correlation coefficient to extend classical bivariate correlation measures to general metric spaces via the use of similarity graphs.

More specifically, let $G_x(V,E_x)$ and $G_y(V,E_y)$ be two graphs constructed from $\{X_i\}_{i\in[N]}$ and $\{Y_i\}_{i\in[N]}$, respectively, to capture within-sample similarity (e.g., via $k$-nearest neighbor graphs or minimum spanning trees). A measure of dependence between $\{X_i\}_{i\in[N]}$ and $\{Y_i\}_{i\in[N]}$ is then defined as
\[
  \Gamma
  = \sum_{i,j\in[N]} \ind\big((i,j)\in E_x \text{ and } (i,j)\in E_y\big).
\]
Under $H_0$ in  \eqref{eq:null-independence}, we can show that $\Gamma$ has the same distribution as
\[
  Q_{\sf \Gamma}
  = \sum_{i,j\in[N]} c_{ij}^{\sf \Gamma} \, a_{\pi(i)\pi(j)}^{\sf \Gamma},
\]
where
\[
  c_{ij}^{\sf \Gamma} := \ind\big((i,j)\in E_x\big),
  \qquad
  a_{ij}^{\sf \Gamma} := \ind\big((i,j)\in E_y\big).
\]
Introduce $\widetilde \C = [\tilde c_{ij}]$ and $\widetilde \A = [\tilde a_{ij}]$ to be the doubly centered versions of $[c_{ij}^{\sf\Gamma}]$ and $[a_{ij}^{\sf\Gamma}]$, defined by
\[
  \tilde c_{ij}
  := c_{ij}^{\sf\Gamma}
     - \frac{1}{N}\sum_{k=1}^N c_{ik}^{\sf\Gamma}
     - \frac{1}{N}\sum_{k=1}^N c_{kj}^{\sf\Gamma}
     + \frac{1}{N^2} \sum_{k,\ell=1}^N c_{k\ell}^{\sf\Gamma},
\]
and similarly for $\tilde a_{ij}$.  In addition, let $\widetilde{\sf d}_x(i)$ and $\widetilde{\sf d}_y(i)$ denote the centered degrees of vertex $i$ in $G_x$ and $G_y$, respectively:
\[
  \widetilde{\sf d}_x(i)
  := {\sf d}_x(i) - \frac{1}{N}\sum_{k=1}^N {\sf d}_x(k),
  \qquad \text{with }
  {\sf d}_x(i) := \big| \{ j : (i,j)\in E_x \} \big|,
\]
and an analogous definition for $\widetilde{\sf d}_y(i)$.

Applying Corollary~\ref{cor:uncentered2} below then yields the following tail bound
\[
  \Prb\Big( Q_{\sf \Gamma} - \Er[Q_{\sf \Gamma}] \ge t \Big)
  \le 
  \exp\Big( - \frac{K t^2}{{\sf V}_a + {\sf B}_a t} \Big)
  + 
  \exp\Big( - \frac{K t^2}{{\sf V}_d + {\sf B}_d t} \Big),
\]
where
\begin{align*}
    {\sf V}_a 
    &:= \frac{1}{N^3}
        \Big( \sum_{i=1}^N \widetilde{\sf d}_x(i)^2 \Big)
        \Big( \sum_{i=1}^N \widetilde{\sf d}_y(i)^2 \Big),\\[0.6ex]
    {\sf B}_a 
    &:= \frac{1}{N}
        \max_{i\in[N]} \big| \widetilde{\sf d}_x(i) \big|
        \max_{i\in[N]} \big| \widetilde{\sf d}_y(i) \big|,\\[0.6ex]
    {\sf V}_d 
    &:= \frac{1}{N^5}
        \Big( \sum_{i=1}^N \widetilde{\sf d}_x(i)^2 \Big)
        \Big( \sum_{i=1}^N \widetilde{\sf d}_y(i)^2 \Big) \\[0.6ex]
    &\quad
      + \frac{1}{N^2}
        \Big(
            |E_x|
            - \frac{1}{N}\sum_{i=1}^N \widetilde{\sf d}_x(i)^2
            - \frac{2|E_x|^2}{N^2}
        \Big)
        \Big(
            |E_y|
            - \frac{1}{N}\sum_{i=1}^N \widetilde{\sf d}_y(i)^2
            - \frac{2|E_y|^2}{N^2}
        \Big),\\[0.6ex]
    {\sf B}_d
    &:= \max_{\sigma\in \mathcal{S}_N} \Big\| \tilde{\C} \circ \tilde{\A}^{\sigma} \Big\|_{\sf op}.
\end{align*}
\end{example}

\subsection{An example in causal inference}\label{sec:ci}

Regression adjustment is widely used to improve efficiency in randomized experiments. \cite{lei2021Regression} studied this method in completely
randomized experiments with a diverging number of covariates and derived a decomposition of the regression-adjusted estimator
$\hat\tau_{\mathrm{adj}}$ proposed by \citet{lin2013agnostic}.
In the decomposition, the bias terms appear in quadratic forms.

More specifically, consider a completely randomized experiment with
$N$ units and a covariate matrix $\mathbf X\in\mathbb{R}^{N\times p}$.  For each treatment arm $\omega\in\{0,1\}$, let $\bm Y(\omega)\in\mathbb{R}^N$ denotes the
potential outcome vector, and let $\bm e(\omega)=(e_1(\omega),\dots,e_N(\omega))^\top$ be the
population regression residuals obtained by regressing $\bm Y(\omega)$ on the covariates. Let
\[
  \mathbf H := \mathbf X (\mathbf X^\top \mathbf X)^{-1} \mathbf X^\top = [h_{ij}] \in \mathbb{R}^{N\times N}
\]
be the hat matrix, and define
\[
  \mathbf Q(\omega)= \big[ q_{ij}(\omega) \big]_{i,j\in[N]} := \mathbf H\,\mathrm{diag}\big(\bm e(\omega)\big),
\]
where $\operatorname{diag}(\bm e(\omega))$ denotes the $N\times N$ diagonal matrix with diagonal entries $e_1(\omega),\dots,e_N(\omega)$. Let 
\[
T_\omega \subset [N] ~~\text{with a fixed size } N_\omega := |T_{\omega}| ~~\text{and proporstion } p_\omega := N_\omega/N
\]
be the set of units assigned to the treatment level $\omega$.  The bias of $\hat\tau_{\mathrm{adj}}$ can
then be written as a function of
\[
  S_{T_\omega} := \sum_{i,j \in T_\omega} q_{ij}(\omega)
\]
for each arm $\omega\in\{0,1\}$.

Under complete randomization, $S_{T_\omega}$ has the same distribution as the following DIPS: 
\[
  Q_{T_\omega}
  = \sum_{i,j \in [N]} c_{ij}^{\sf T_\omega} \, a_{\pi(i)\pi(j)}^{\sf T_\omega},
  \qquad \text{with }
  c_{ij}^{\sf T_\omega} := \ind( i \le N_\omega,\ j \le N_\omega ) \text{ and }
  a_{ij}^{\sf T_\omega} := q_{ij}(\omega).
\]
Following \cite{lei2021Regression}, we assume $\mathbf X$ to be column-centered, which ensures that both $\mathbf{H}$ and $\mathbf Q(\omega)$ are doubly centered. Applying Corollary~\ref{cor:uncentered2}, we can show that there exists a universal constant $K>0$ such that, for any $t>0$, 
\begin{align*}
  \Prb\big(Q_{T_\omega} - \Er [Q_{T_\omega}] \ge t\big)
  &\le
  \exp\left(
    - \frac{K t^2}{ p_1p_0 \|\mathbf Q(\omega)\|_{\sf F}^2
       +   \max\{p_1^2, p_0^2\} 
         \|\mathbf Q(\omega)\|_{\sf op}t}
  \right) 
  \\
  &\le
  \exp\left(
    - \frac{K t^2}{ p_1p_0 \sum_{i=1}^N e_i(\omega)^2 \max_{1\leq i\leq N} h_{ii} 
       +   \max\{p_1^2, p_0^2\} 
         \max_{1\leq i\leq N} | e_i(\omega) | t}
  \right) ,
\end{align*}
where the second inequality follows from
$
\|\Q(\omega)\|_{\sf F}^2 
= \sum_{i=1}^N e_i(\omega)^2 h_{ii} 
\leq \sum_{i=1}^N e_i(\omega)^2 \max_{1\leq i\leq N} h_{ii} $ by \citet[][Lemma A.10]{lei2021Regression} and $\|\mathbf Q(\omega)\|_{\sf op} 
\leq \| \mathbf H \|_{\sf op} \| \mathrm{diag}\big(\bm e(\omega)\big) \|_{\sf op}
= \max_{1\leq i\leq N} | e_i(\omega) |$. 

\section{Main results}\label{sec:main}

We begin by introducing some necessary terminology for tensors and DIPS.  
For any general fourth-order tensor $\W = [w(i,j,k,\ell)] \in \Rb^{N\times N\times N\times N}$, define
\begin{align*} 
w(i,j,k,\cdot) &:= \frac{1}{N}\sum_{\ell=1}^{N} w(i,j,k,\ell),\quad w(i,j,\cdot,\cdot) := \frac{1}{N^2} \sum_{k,\ell=1}^{N}w(i,j,k,\ell),\\ 
w(i,\cdot,\cdot,\cdot)&:= \frac{1}{N^3} \sum_{j,k,\ell=1}^{N} w(i,j,k,\ell),\quad w(\cdot,\cdot,\cdot,\cdot) := \frac{1}{N^4} \sum_{i,j,k,\ell=1}^{N} w(i,j,k,\ell), 
\end{align*}
with other partial averages defined analogously.

We call a fourth-order tensor $\W$ \emph{partly degenerate} if at least one of its partial averages is zero. A partly degenerate kernel is further said to be
\emph{degenerate} if all partial averages are zero, i.e., for all $i,j,k,\ell\in[N]$,
\[
  w(i,j,k,\cdot)
  = w(i,j,\cdot,\ell)
  = w(i,\cdot,k,\ell)
  = w(\cdot,j,k,\ell)
  = 0.
\]
In what follows, we reserve the notation $\D = [d(i,j,k,\ell)]$ for partly degenerate tensors and use $\W$ to denote a general (not necessarily degenerate) tensor.

A statistic $Q_w$ of the form \eqref{eq:dips} can always be decomposed into the sum of a single-indexed term and a degenerate DIPS, analogous to the classical Hoeffding decomposition for U-statistics. More precisely,
\begin{align}\label{eq:decomp}
&\sum_{i,j\in[N]} w(i,j,\pi(i),\pi(j))
  - \Er\Big[\sum_{i,j\in[N]} w(i,j,\pi(i),\pi(j))\Big] \notag\\
=& 
  N \left(\sum_{i=1}^N a_w\big(i,\pi(i)\big)-\Er\left[\sum_{i=1}^N a_w\big(i,\pi(i)\big)\right]\right)\notag
  \\
&\quad+ \sum_{i,j\in[N]} d_w\big(i,j,\pi(i),\pi(j)\big)-\Er\left[\sum_{i,j\in[N]} d_w\big(i,j,\pi(i),\pi(j)\big)\right],
\end{align}
where, for any $i,j,k,\ell\in[N]$,
\begin{align}\label{eq:aw}
  a_w(i,j)
  :=&\;
  w(i,\cdot,j,\cdot)
  - w(i,\cdot,\cdot,\cdot)
  - w(\cdot,\cdot,j,\cdot)
  + w(\cdot,\cdot,\cdot,\cdot) \notag\\
  &\quad
  + w(\cdot,i,\cdot,j)
  - w(\cdot,i,\cdot,\cdot)
  - w(\cdot,\cdot,\cdot,j)
  + w(\cdot,\cdot,\cdot,\cdot),
\end{align}
and
\begin{align}\label{eq:dw}
  d_w(i,j,k,\ell)
  :=&\;
  w(i,j,k,\ell)
  - w(i,\cdot,\cdot,\cdot)
  - w(\cdot,j,\cdot,\cdot)
  - w(\cdot,\cdot,k,\cdot)
  - w(\cdot,\cdot,\cdot,\ell) \notag\\[0.4ex]
  &\quad
  + w(i,j,\cdot,\cdot)
  + w(i,\cdot,k,\cdot)
  + w(i,\cdot,\cdot,\ell)
  + w(\cdot,j,k,\cdot)
  + w(\cdot,j,\cdot,\ell)
  + w(\cdot,\cdot,k,\ell)\notag\\[0.4ex]
  &\quad
  - w(i,j,k,\cdot)
  - w(i,j,\cdot,\ell)
  - w(i,\cdot,k,\ell)
  - w(\cdot,j,k,\ell)\notag\\[0.4ex]
  &\quad
  + w(\cdot,\cdot,\cdot,\cdot).
\end{align}

In light of the decomposition in \eqref{eq:decomp}, degenerate DIPS plays the central role in determining the stochastic behavior of general DIPS. The next two theorems provide the decoupling and randomization inequalities for such degenerate DIPS. 

\begin{theorem}[Combinatorial decoupling] \label{thm:decoupling}
Let $N\geq 20$ and $\mathbf{D}=[d(i,j,k,\ell)]\in\Rb^{N\times N\times N\times N}$ be a degenerate tensor. Let $I$ and $J$ be independently and uniformly distributed over all subsets of $[N]$ with size $N_1 = \lceil N/2\rceil$, where $\lceil \cdot\rceil$ represents the ceiling function. Let $\pi_1$ be uniformly distributed over all bijections from $I$ to $J$ and $\pi_2$ be uniformly distributed over all bijections from $I^c$ to $J^c$, independent of $\pi_1$. Then, for any $\lambda\geq 0$,
$$\begin{aligned}
&\Er\left[\exp\left(\lambda\left(\sum_{i\ne j} d(i,j,\pi(i),\pi(j))-\Er\left[\sum_{i\ne j} d(i,j,\pi(i),\pi(j))\right]\right)\right)\right]\\
\leq& \exp\left(50\tilde V\lambda^2\right)\Big(\Er \Big[\exp\Big\{9\lambda\Big(\sum_{i\in I,j\in I^c}\tilde{d}_{I,J}(i,j,\pi_1(i),\pi_2(j))\Big)\Big\}\Big]\Big)^{1/2},
\end{aligned}$$
where $\tilde d_{I,J}$ is the centralized version of $d$ restricted over the arrays $(i,j,k,\ell)\in I\times I^c\times J\times J^c$, explicitly defined in \eqref{eq:tilded-IJ}, and 
\begin{equation}
\label{eq:off-diagonal variance}
\tilde V := \frac{1}{N^2}\sum_{i,j,k,\ell\in[N]}d(i,j,k,\ell)^2.
\end{equation}
\end{theorem}

\begin{theorem}[Combinatorial randomization]\label{thm:randomization}
Let $N,M$ be two positive integers and $\pi,\tau$ be two independent random permutations uniformly distributed over $\mathcal{S}_N$ and $\mathcal{S}_M$, respectively. Consider 
\[
\D=[d(i,j,k,\ell)]\in \Rb^{N\times M \times N \times M} 
\]
to be a (possibly random) tensor independent of $(\pi,\tau)$ satisfying $d(\cdot,j,\cdot,\ell)=d(i,\cdot,k,\cdot)=0$ for any $i,k\in[N]$ and $j,\ell\in[M]$. 
Then, for any $\lambda \geq 0$,
\begin{align*}
&\Er\Bigg\{\exp\Bigg(\lambda\sum_{\substack{i\in [N]\\j\in[M]}}  d(i,j,\pi(i),\tau(j))\Bigg)\Bigg\}
\leq \Er\Bigg\{\exp\Bigg(12\lambda\sum_{\substack{i\in [N]\\j\in[M]}}  d(i,j,\pi(i),\tau(j)) G_i G'_j\Bigg)\Bigg\},
\end{align*}
where $G_1, G_2, \ldots,G_1', G_2', \ldots$ are independent standard Gaussian random variables further independent of the system.
\end{theorem}

The power of Theorems~\ref{thm:decoupling} and~\ref{thm:randomization} can be appreciated through the following result, which provides the \emph{first} concentration inequality for general DIPS.

\begin{theorem}\label{thm:main}
Let $N\geq 20$ and $\mathbf{D}=[d(i,j,k,\ell)]\in\Rb^{N\times N\times N\times N}$ be a fixed degenerate tensor. Introduce
\begin{align}\label{eq:BV}
B:=\max_{\sigma\in \mathcal{S}_N}\Big\|[d(i,j,\sigma(i),\sigma(j))]_{i,j}\Big\|_{\sf op}~~{\rm and}~~V := \frac{1}{N}\sum_{i,k\in[N]}\xi(i,k)^2+\frac{1}{N^2}\sum_{\substack{i,j,k,\ell\in[N]\\i\neq j, k\neq l}}d(i,j,k,\ell)^2,
\end{align}
where $[d(i,j,\sigma(i),\sigma(j))]_{i,j}$ denotes the $N\times N$ matrix whose $(i,j)$-th entry is $d(i,j,\sigma(i),\tilde{\sigma}(j))$ and \[
\xi(i,k) := d(i,i,k,k)-\frac{1}{N}\sum_{k'=1}^N d(i,i,k',k')-\frac{1}{N}\sum_{i'=1}^N d(i',i',k,k)+\frac{1}{N^2}\sum_{i',k'\in [N]}d(i',i',k',k').
\]
Recall $Q_d$ defined in \eqref{eq:degeneratedips}. 
Then, for all $\lambda\in[0,1/5400B]$, 
\[
\begin{aligned} 
&\Er\Big[\exp(\lambda (Q_d-\Er[Q_d]))\Big]\leq \exp\left(\frac{100000V\lambda^2}{1-5400B\lambda}\right),\\
\end{aligned}
\]
 which implies that for any $t>0$,
\[
\Prb\Big(Q_d-\Er[Q_d]\geq t\Big)\leq\exp\left(-\frac{t^2}{400000V+10800Bt}\right).
\]
\end{theorem}
Combining Theorem~\ref{thm:main}, the decomposition~\eqref{eq:decomp}, and a Bernstein-type inequality for Hoeffding's statistics \citep[Theorem~3.7]{han2024introduction}, originally due to \citet{chatterjee2007Stein}, we obtain the following corollary for $Q_w$ with a general, possibly non-degenerate tensor.

\begin{corollary}
\label{cor:uncentered2}
 Let $N\geq 20$ and $\W=[w(i,j,k,\ell)]\in\Rb^{N\times N\times N\times N}$ be a fixed tensor. Recall $Q_w$ defined in \eqref{eq:dips}. 
Then, there exists a universal constant $K>0$ such that for any $t>0$,
$$
\Prb(Q_w-\Er[Q_w] \geq t)
\leq \exp\left(-\frac{Kt^2}{V_a+B_at}\right)+ \exp\left(-\frac{Kt^2}{V_d+B_dt}\right),
$$
where
\begin{align*}
B_a:=N\max_{i,j\in[N]}\Big|a_w(i,j)\Big|,~~V_a:=N\sum_{i,j\in[N]} a_w(i,j)^2, ~~B_d:=\max_{\sigma\in \mathcal{S}_N}\Big\|\Big[d_w(i,j,\sigma(i),\sigma(j))\Big]_{i,j}\Big\|_{\sf op},\\
\xi_w(i,k) = d_w(i,i,k,k)-\frac{1}{N}\sum_{k'=1}^N d_w(i,i,k',k')-\frac{1}{N}\sum_{i'=1}^N d_w(i',i',k,k)+\frac{1}{N^2}\sum_{i',k'\in [N]}d_w(i',i',k',k')\\
{\rm and}~~V_d:=\frac{1}{N}\sum_{i,k\in[N]}\xi_w(i,k)^2+\frac{1}{N^2}\sum_{\substack{i,j,k,\ell\in[N]\\i\neq j, k\neq l}}d_w(i,j,k,\ell)^2,
\end{align*}
with $a_w(\cdot,\cdot)$ and $d_w(\cdot,\cdot,\cdot,\cdot)$ introduced in \eqref{eq:aw} and \eqref{eq:dw}, respectively.
\end{corollary}

\section{Proofs of the main results} \label{sec:proofs}

In the following, recall that $G_1,G_2, \ldots,G_1',G_2', \ldots,$ are independent standard Gaussian random variables, further independent of the system. Recall that $\bigl[\cdot\bigr]_{i,j}$ denotes the matrix whose $(i,j)$-th entry is given by the enclosed expression, where $i$ and $j$ range over the prescribed index sets (not necessarily $[N]$). For a set $I \subseteq [N]$, we write $|I|$ for its cardinality and $I^c := [N]\setminus I$ for its complement. Additionally, for a permutation (or, more generally, a mapping) $\pi$, we denote $\pi(I) := \{\pi(i) \mid i \in I\}$. For a matrix $A=[a(i,j)] \in \mathbb{R}^{N\times N}$, define 
\begin{equation}
\label{eq:1dstatistics}
f(\pi) := \sum_{i=1}^N a(i,\pi(i)).
\end{equation}
When the dependence on $\pi$ needs to be emphasized, we write $f(\pi)$; otherwise, we write $f$ for simplicity.
\subsection{Auxiliary lemmas}

\begin{lemma}
\label{lem:1d}
Let
\(
\A=[a(i,j)]\in \Rb^{N\times N} 
\)
be a (possibly random) matrix independent of $\pi$ and satisfying $\sum_{i,j\in[N]}a(i,j)=0$. We then have, for any $\lambda\in \mathbb{R}$,
$$
\Er\left[\exp\left(\lambda\sum_{i=1}^N a(i,\pi(i))\right)\right]
\leq \Er\left[\exp\left(2\sqrt{3}\lambda\sum_{i=1}^N  a(i,\pi(i))G_i\right)\right].
$$
\end{lemma}

\begin{lemma}
\label{lem:1d_MGF}
Let $\A=[a(i,j)] \in \mathbb{R}^{N \times N}$ be a fixed matrix. We then have,  for any $|\lambda|< \sqrt{2}/B_A$,
\[
\Er\left[\exp\left(\lambda\sum_{i=1}^N a(i,\pi(i))G_i\right)\right]\leq \exp\left(\frac{V_A\lambda^2-V_AB_A^2\lambda^4/4}{2(1-B_A^2\lambda^2/2)}\right),
\]
where 
\begin{equation}
\label{eq:defineBVforA}
B_A := \max_{i,j\in[N]}|a(i,j)|~~ \text{ and }~~ V_A := \frac{1}{N}\sum_{i,j\in[N]}a(i,j)^2.
\end{equation}
\end{lemma}
\begin{lemma}
\label{lem:2d_MGF}
Under the setting of Theorem \ref{thm:randomization}, we have, for any $|\lambda|< 0.64/B_{\mathrm{rect}}$,\[\Er\left[\exp\left(\lambda{\sum_{\substack{i\in[N]\\ j\in[M]}} d(i,j,\pi(i),\tau(j))G_iG'_j}\right)\right]\leq \exp\left(\frac{\tilde{V}_{\mathrm{rect}}\lambda^2}{2(1-\frac{B_{\mathrm{rect}}\lambda}{0.64})}\right),\]
where \begin{equation}
\label{eq:defineBVrect}
B_{\mathrm{rect}} := \max_{\substack{\sigma\in \mathcal{S}_N\\\tilde{\sigma}\in \mathcal{S}_M}}\Big\|[d(i,j,\sigma(i),\tilde{\sigma}(j))]_{i,j}\Big\|_{\sf op} ~~\text{ and }~~ \tilde{V}_{\mathrm{rect}} := \frac{1}{NM}\sum_{\substack{i,k\in[N]\\ j,\ell\in [M]}}d(i,j,k,\ell)^2.
\end{equation}
\end{lemma}
\begin{lemma}
\label{lem:Bennet_MGF}
Let \(X\) be a real random variable satisfying, for any \(\lambda\ge 0\),
\[
\Er\big[e^{\lambda X}\big]\le\exp\big(C\,\lambda^{2}e^{4\lambda}\big)
\]
for some constant \(C>0\). Then for any \(t\ge 0\),
\[
\Prb\Big(X\ge t\Big)
\le
\exp\left(-\frac{t}{12}\log\left(1+\frac{t}{C}\right)\right).
\]
\end{lemma}
\begin{lemma}
\label{lem:Bennet_1d_lem1}
Let $\A=[a(i,j)] \in \mathbb{R}^{N \times N}$ be a fixed matrix whose entries are in $[0,1]$ and $f$ be defined as in \eqref{eq:1dstatistics}. Then for any $\lambda\geq 0$,
\[
\Er\Big[\exp\Big\{\lambda(f-\Er[f])\Big\}\Big]\leq \exp\Big(2\,\Er[f]\lambda(e^{2\lambda}-1)\Big).
\]
\end{lemma}
\begin{lemma}
\label{lem:Bennet_1d_lem2}
Let $\A=[a(i,j)] \in \mathbb{R}^{N \times N}$ be a fixed matrix whose entries are in $[-1,1]$ and $f$ be defined as in \eqref{eq:1dstatistics}. 
Then for any $|\lambda|\leq 1/3$,
\[
\Er\Big[\exp\Big\{\lambda(f-\Er[f])\Big\}\Big]\leq \exp(12V_A\lambda^2),
\]    
where $V_A$ is introduced in \eqref{eq:defineBVforA}.
\end{lemma}
\begin{lemma}
\label{lem:Bennet_1d}
Fix $\A$ as in Lemma \ref{lem:Bennet_1d_lem2} and let $f$ be defined as in \eqref{eq:1dstatistics}. Then for any $\lambda\geq 0$,
$$
\Er\Big[\exp\Big\{\lambda (f-\Er[f])\Big\}\Big]
\leq \exp(54V_A\lambda^2e^{4\lambda}).
$$
Consequently, for any $t>0$,
$$
\Prb\Big(f-\Er[f] \geq t\Big)
\leq \exp\left\{-\frac{t}{12}\log{\left(1+\frac{t}{54V_A}\right)}\right\},
$$ 
where $V_A$ is introduced in \eqref{eq:defineBVforA}.
\end{lemma}
\begin{lemma}
\label{lem:Bennet_2d_decoupled}
Let $\pi$ and $\tau$ be defined as in Theorem \ref{thm:randomization}. Let $\mathbf{A},\mathbf{C} \in \mathbb{R}^{N \times M}$ be two fixed matrices, and assume that $\C$ is doubly centered.  Denote the generalized combinatorial bilinear form
$$
 Q_{\sf BF}=\sum_{\substack{i\in [N]\\j\in[M]}} c_{i j} a_{\pi(i)\tau(j)}. 
$$
Then for any $\lambda\ge0$,
$$
\Er\Bigg[\exp\Bigg\{\frac{\lambda ( Q_{\sf BF}-\Er[ Q_{\sf BF}])}{\nu_{\sf rect}}\Bigg\}\Bigg]
\leq \exp\left(\frac{54}{N\nu_{\sf rect}^2}\norm{\C}_{\sf op}^2\norm{\A}^2_{\sf F}\lambda^2e^{4\lambda}\right),
$$
which implies that for any $t\geq0$,
\[
\Prb\Big( Q_{\sf BF}-\Er[ Q_{\sf BF}] \geq t\Big)
\leq \exp\left\{-\frac{t}{12\nu_{\sf rect}}\cdot\log{\left(1+\frac{\nu_{\sf rect} t}{\frac{54}{N}\norm{\C}_{\sf op}^2\norm{\A}_{\sf F}^2}\right)} \right\},
\]
where 
\[
\nu_{\sf rect} := \max_{\substack{i,k\in[N]\\ \sigma,\tilde{\sigma}\in\mathcal{S}_M}}\left|\sum_{j=1}^{M} c_{i\sigma(j)}a_{k\tilde{\sigma}(j)}\right|.
\]
\end{lemma}
\begin{lemma}
\label{lem:decoupling_zero_diagonal}Under the setting of Theorem \ref{thm:Bennett}, denote $d(i,j,k,\ell) = a_{ij}c_{k\ell }$ for any $i,j,k,\ell\in [N]$ and let $I$,  $J$, $\pi_1$, $\pi_2$ be defined as in Theorem \ref{thm:decoupling}. Then for any convex function $\varphi$,
\begin{align*}
   &\quad\Er\left[\varphi\left(\sum_{i\ne j} d(i,j,\pi(i),\pi(j))\right)\right]\\
   &\leq\frac{1}{2}\Er \left[\varphi\left(2\alpha\sum_{i\in I,j\in I^c}\tilde{d}_{I,J}(i,j,\pi_1(i),\pi_2(j))\right)\right]+ \frac{1}{2}\Er \left[\varphi\left(2\beta\sum_{i\ne j} d(i,j,\pi(j),\pi(i))\right)\right],
\end{align*}
where 
\begin{equation}\label{DefAlphaBeta}
    \alpha := \frac{N(N-1)(N-2)(N-3)}{(N_1-1)(N-N_1-1)\big(N^2 - 3N + 1\big)}, ~~\beta := \frac{1}{N^2 - 3N + 1}, 
\end{equation}
and $\tilde d_{I,J}$ was introduced in \eqref{eq:tilded-IJ}.
\end{lemma}

\subsection{Proof of Theorem \ref{thm:decoupling}}
Introduce the explicit definition of $\tilde{d}_{I,J}$: 
\begin{align}\label{eq:tilded-IJ}
&\tilde{d}_{I,J}(i,j,k,\ell):=d(i,j,k,\ell)\notag\\
&-\frac{1}{|I|}\sum_{i\in I} d(i,j,k,\ell)-\frac{1}{|I^{c}|}\sum_{j\in I^c} d(i,j,k,\ell)-\frac{1}{|J|}\sum_{k\in J} d(i,j,k,\ell)-\frac{1}{|J^{c}|}\sum_{\ell\in J^c} d(i,j,k,\ell) \notag\\
&+\frac{1}{|I||I^c|}\sum_{i\in I,j\in I^c} d(i,j,k,\ell)+\frac{1}{|I||J|}\sum_{i\in I,k\in J} d(i,j,k,\ell)+\frac{1}{|I||J^c|}\sum_{i\in I,\ell\in J^c} d(i,j,k,\ell) \notag\\
&+\frac{1}{|I^c||J|}\sum_{j\in I^c,k\in J} d(i,j,k,\ell)+\frac{1}{|I^c||J^c|}\sum_{j\in I^c,\ell\in J^c} d(i,j,k,\ell)+\frac{1}{|J||I^c|}\sum_{k\in J,\ell\in J^c} d(i,j,k,\ell) \notag\\
&-\frac{1}{|I||I^c||J|}\sum_{i\in I,j\in I^c,k\in J} d(i,j,k,\ell)-\frac{1}{|I||I^c||J^c|}\sum_{i\in I,j\in I^c,\ell\in J^c} d(i,j,k,\ell) \notag\\
&-\frac{1}{|I||J||J^c|}\sum_{i\in I,k\in J,\ell\in J^c} d(i,j,k,\ell)-\frac{1}{|I^c||J||J^c|}\sum_{j\in I^c,k\in J,\ell\in J^c} d(i,j,k,\ell) \notag\\
&+\frac{1}{|I||I^c||J||J^c|}\sum_{i\in I,j\in I^c,k\in J,\ell\in J^c}d(i,j,k,\ell).
\end{align}

\begin{proof}[Proof of Theorem \ref{thm:decoupling}]
We have 
$$\begin{aligned}
        &\quad\Er\left[\left.\sum_{i\in I,j\in I^c}\tilde{d}_{I,\pi(I)}(i,j,\pi(i),\pi(j))\right|\pi\right]\\
        &=\Er\left[\sum_{i\in I,j\in I^c}d(i,j,\pi(i),\pi(j))-\sum_{i\in I,j\in I^c}\frac{1}{|I|}\sum_{k\in \pi(I)} d(i,j,k,\pi(j))\right.\\
        &\quad-\left.\left.\sum_{i\in I,j\in I^c}\frac{1}{|I^{c}|}\sum_{\ell\in \pi(I^{c})} d(i,j,\pi(i),\ell)+\sum_{i\in I,j\in I^c}\frac{1}{|I||I^c|}\sum_{k\in \pi(I),\ell\in\pi(I^c)} d(i,j,k,\ell)\right|\pi\right]\\
        &=\Er\left[\left.\sum_{i\in I,j\in I^c}d(i,j,\pi(i),\pi(j))\right|\pi\right]-\Er\left[\left.\sum_{i\in I,j\in I^c}\frac{1}{|I|}\sum_{k\in \pi(I)} d(i,j,k,\pi(j))\right|\pi\right]\\
        &\quad-\Er\left[\left.\sum_{i\in I,j\in I^c}\frac{1}{|I^{c}|}\sum_{\ell\in \pi(I^{c})} d(i,j,\pi(i),\ell)\right|\pi\right]+\Er\left[\left.\sum_{i\in I,j\in I^c}\frac{1}{|I||I^c|}\sum_{k\in \pi(I),\ell\in\pi(I^c)} d(i,j,k,\ell)\right|\pi\right].
    \end{aligned}
$$
For the first and second terms on the right hand side, we have
$$
\Er\left[\left.\sum_{i\in I,j\in I^c}d(i,j,\pi(i),\pi(j))\right|\pi\right] = \frac{\binom{N-2}{N_1-1}}{\binom{N}{N_1}}\sum_{i\neq j}d(i,j,\pi(i),\pi(j)) = \frac{N_1(N-N_1)}{N(N-1)}\sum_{i\neq j}d(i,j,\pi(i),\pi(j)), 
$$
and
\begin{align*}
&\quad\Er\left[\sum_{i\in I, j\in I^{c}}\frac{1}{|I|}\sum_{k\in \pi(I)} d(i,j,k,\pi(j))\middle|\pi\right] \\
&= \frac{1}{N_1}\sum_{i\neq j}\sum_{k=1}^{N}
d(i,j,k,\pi(j))\,
\Prb\left(i\in I, j\notin I, \pi^{-1}(k)\in I\right) \\
&= \frac{1}{N_1}\sum_{i\ne j}\Bigg[
\sum_{\pi^{-1}(k)\notin\{i,j\},k}
d(i,j,k,\pi(j))\frac{\binom{N-3}{N_1-2}}{\binom{N}{N_1}}
+
d\big(i,j,\pi(i),\pi(j)\big)\frac{\binom{N-2}{N_1-1}}{\binom{N}{N_1}}
+
d\big(i,j,\pi(j),\pi(j)\big)\cdot 0
\Bigg] \\
&=
\frac{\binom{N-3}{\,N_1-2\,}}{N_1\binom{N}{\,N_1\,}}
\sum_{i\ne j}\sum_{k=1}^{N} d(i,j,k,\pi(j))
+\left(
\frac{\binom{N-2}{\,N_1-1\,}}{N_1\binom{N}{\,N_1\,}}
-\frac{\binom{N-3}{\,N_1-2\,}}{N_1\binom{N}{\,N_1\,}}
\right)\sum_{i\ne j} d(i,j,\pi(i),\pi(j))
\\
&\quad-\frac{\binom{N-3}{\,N_1-2\,}}{N_1\binom{N}{\,N_1\,}}
\sum_{i\ne j} d(i,j,\pi(j),\pi(j)) \\
&=\frac{(N-N_1)(N-N_1-1)}{N(N-1)(N-2)} \sum_{i\ne j} d(i,j,\pi(i),\pi(j))
-
\frac{(N_1-1)(N-N_1)}{N(N-1)(N-2)} \sum_{i\ne j} d(i,j,\pi(j),\pi(j))\\
&=\frac{(N-N_1)(N-N_1-1)}{N(N-1)(N-2)} \sum_{i\ne j} d(i,j,\pi(i),\pi(j))+
\frac{(N_1-1)(N-N_1)}{N(N-1)(N-2)} \sum_{i} d(i,i,\pi(i),\pi(i)).
\end{align*}
Similarly, for the third term,
\begin{align*}
&\quad\Er\left[\sum_{i\in I, j\in I^{c}}\frac{1}{|I^c|}\sum_{\ell\in \pi(I^c)} d(i,j,\pi(i),\ell)\middle|\pi\right] \\
&=\frac{N_1(N_1-1)}{N(N-1)(N-2)} \sum_{i\ne j} d(i,j,\pi(i),\pi(j))
-
\frac{(N-N_1-1)N_1}{N(N-1)(N-2)} \sum_{i\ne j} d(i,j,\pi(i),\pi(i))\\
&=\frac{N_1(N_1-1)}{N(N-1)(N-2)} \sum_{i\ne j} d(i,j,\pi(i),\pi(j))+
\frac{(N-N_1-1)N_1}{N(N-1)(N-2)} \sum_{i=1}^N d(i,i,\pi(i),\pi(i)).
\end{align*}
For the fourth term, we have 
\begin{align*}
&\quad\Er\left[\sum_{i\in I, j\in I^{c}}\frac{1}{|I||I^{c}|}
\sum_{k\in \pi(I), \ell\in \pi(I^{c})} d(i,j,k,\ell)\middle|\pi\right] \\[2mm]
&= \frac{1}{N_1(N-N_1)}\sum_{i\ne j}\sum_{k\ne \ell}
d(i,j,k,\ell)\Prb\left(i\in I, j\notin I, \pi^{-1}(k)\in I, \pi^{-1}(\ell)\notin I\right) \\
&= \frac{1}{N_1(N-N_1)}\sum_{i\ne j}\Bigg[
\frac{\binom{N-4}{N_1-2}}{\binom{N}{N_1}}
\sum_{\substack{k\notin\{\pi(i),\pi(j)\}\\\ell\notin\{\pi(i),\pi(j)\}, \ell\ne k}}
d(i,j,k,\ell)
+
\frac{\binom{N-3}{N_1-1}}{\binom{N}{N_1}}
\sum_{\substack{\ell\notin\{\pi(i),\pi(j)\}}}
d(i,j,\pi(i),\ell) \\[1mm]
&\quad+
\frac{\binom{N-3}{N_1-2}}{\binom{N}{N_1}}
\sum_{\substack{k\notin\{\pi(i),\pi(j)\}}}
d(i,j,k,\pi(j))\quad+
\frac{\binom{N-2}{N_1-1}}{\binom{N}{N_1}}
d(i,j,\pi(i),\pi(j))
\Bigg] \\
&= \frac{1}{N_1(N-N_1)}\sum_{i\ne j}\Bigg[
\frac{N_1(N_1-1)(N-N_1)(N-N_1-1)}{N(N-1)(N-2)(N-3)}
\sum_{\substack{k\notin\{\pi(i),\pi(j)\}\\ \ell\notin\{\pi(i),\pi(j)\}, \ell\ne k}}
d(i,j,k,\ell) \\
&\quad+\frac{N_1(N-N_1)(N-N_1-1)}{N(N-1)(N-2)}
\sum_{\substack{\ell\notin\{\pi(i),\pi(j)\}}}
d(i,j,\pi(i),\ell)
+
\frac{N_1(N_1-1)(N-N_1)}{N(N-1)(N-2)}
\sum_{\substack{k\notin\{\pi(i),\pi(j)\}}}
d(i,j,k,\pi(j)) \\[1mm]
&\quad+\frac{N_1(N-N_1)}{N(N-1)}d(i,j,\pi(i),\pi(j))
\Bigg]\\
&= -\frac{1}{N_1(N-N_1)}\sum_{i\ne j}\Bigg[
\frac{N_1(N_1-1)(N-N_1)(N-N_1-1)}{N(N-1)(N-2)(N-3)}\big(-d(i,j,\pi(i),\pi(j))-d(i,j,\pi(j),\pi(i))\\
&\quad- 2d(i,j,\pi(i),\pi(i))-2d(i,j,\pi(j),\pi(j))+\sum_k d(i,j,k,k)\big)\\
&\quad+\frac{N_1(N-N_1)(N-N_1-1)}{N(N-1)(N-2)}
\big(d(i,j,\pi(i),\pi(i))+d(i,j,\pi(i),\pi(j))\big)\\
&\quad+
\frac{N_1(N_1-1)(N-N_1)}{N(N-1)(N-2)}
(d(i,j,\pi(i),\pi(j))+d(i,j,\pi(j),\pi(j))) \\
&\quad-\frac{N_1(N-N_1)}{N(N-1)}d(i,j,\pi(i),\pi(j))
\Bigg]\\
&=\frac{(N_1-1)(N-N_1-1)}{N(N-1)(N-2)(N-3)} \sum_{i\ne j} d(i,j,\pi(i),\pi(j))+\frac{(N_1-1)(N-N_1-1)}{N(N-1)(N-2)(N-3)} \sum_{i\ne j} d(i,j,\pi(j),\pi(i))\\
&\quad-\left(\frac{N-N_1-1}{N(N-1)(N-2)}-\frac{2(N_1-1)(N-N_1-1)}{N(N-1)(N-2)(N-3)}\right)
\sum_{i\ne j}d(i,j,\pi(i),\pi(i))\\
&\quad-
\left(\frac{N_1-1}{N(N-1)(N-2)}-\frac{2(N_1-1)(N-N_1-1)}{N(N-1)(N-2)(N-3)}\right)
\sum_{i\ne j}d(i,j,\pi(j),\pi(j)) \\
&\quad -\frac{(N_1-1)(N-N_1-1)}{N(N-1)(N-2)(N-3)} \sum_{i\ne j}\sum_{k=1}^N d(i,j,k,k)\\
&=\frac{(N_1-1)(N-N_1-1)}{N(N-1)(N-2)(N-3)} \sum_{i\ne j} d(i,j,\pi(i),\pi(j))+\frac{(N_1-1)(N-N_1-1)}{N(N-1)(N-2)(N-3)} \sum_{i\ne j} d(i,j,\pi(j),\pi(i))\\
&\quad + \left(\frac1{N(N-1)}-\frac{4(N_1-1)(N-N_1-1)}{N(N-1)(N-2)(N-3)}\right) \sum_{i=1}^N d(i,i,\pi(i),\pi(i))\\
&\quad+\frac{(N_1-1)(N-N_1-1)}{N(N-1)(N-2)(N-3)} \sum_{i,k\in[N]} d(i,i,k,k).
\end{align*}
Therefore,
\begin{align*}
        &\quad\Er\left[\left.\sum_{i\in I,j\in I^c}\tilde{d}_{I,\pi(I)}(i,j,\pi(i),\pi(j))\right|\pi\right]\\
        &=\Er\left[\left.\sum_{i\in I,j\in I^c}d(i,j,\pi(i),\pi(j))\right|\pi\right]-\Er\left[\left.\sum_{i\in I,j\in I^c}\frac{1}{|I|}\sum_{k\in \pi(I)} d(i,j,k,\pi(j))\right|\pi\right]\\
        &\quad-\Er\left[\left.\sum_{i\in I,j\in I^c}\frac{1}{|I^{c}|}\sum_{\ell\in \pi(I^{c})} d(i,j,\pi(i),\ell)\right|\pi\right]+\Er\left[\left.\sum_{i\in I,j\in I^c}\frac{1}{|I||I^c|}\sum_{k\in \pi(I),\ell\in\pi(I^c)} d(i,j,k,\ell)\right|\pi\right]\\
        &=\frac{(N_1-1)(N-N_1-1)\big(N^2 - 3N + 1\big)}{N(N-1)(N-2)(N-3)}\left(\sum_{i\ne j} d(i,j,\pi(i),\pi(j))-\Er\left[\sum_{i\ne j} d(i,j,\pi(i),\pi(j))\right]\right)-\Delta(\pi),
    \end{align*}
where

$$\begin{aligned}
    \Delta(\pi) &= \frac{2(N_1-1)(N-N_1-1)}{N(N-2)(N-3)}\sum_{i=1}^N d(i,i,\pi(i),\pi(i))\\
    &\quad-\frac{(N_1-1)(N-N_1-1)}{N(N-1)(N-2)(N-3)}\sum_{i\ne j} d(i,j,\pi(j),\pi(i))\\
    &\quad -\frac{(N_1-1)(N-N_1-1)}{N(N-1)(N-2)(N-3)} \sum_{i,k\in[N]} d(i,i,k,k)\\
    &\quad-\frac{(N_1-1)(N-N_1-1)\big(N^2 - 3N + 1\big)}{N(N-1)(N-2)(N-3)}\Er\left(\sum_{i\ne j} d(i,j,\pi(i),\pi(j))\right).
\end{aligned}
$$
By the definition of $\alpha$ in \eqref{DefAlphaBeta}, we have
$$       \quad\Er\left[\left.\alpha\left(\sum_{i\in I,j\in I^c}\tilde{d}_{I,\pi(I)}(i,j,\pi(i),\pi(j))+\Delta(\pi)\right)\right|\pi\right]=\sum_{i\ne j} d(i,j,\pi(i),\pi(j))-\Er\left[\sum_{i\ne j} d(i,j,\pi(i),\pi(j))\right].
$$
Since $x\to e^{\lambda x}$ is convex, by Jensen's inequality,
$$
\begin{aligned}
   &\quad\Er\left[\exp\left(\lambda\left(\sum_{i\ne j} d(i,j,\pi(i),\pi(j))-\Er\left[\sum_{i\ne j} d(i,j,\pi(i),\pi(j))\right]\right)\right)\right]\\
   &= \Er \left[\exp\left(\lambda\Er\left[\left.\alpha\left(\sum_{i\in I,j\in I^c}\tilde{d}_{I,\pi(I)}(i,j,\pi(i),\pi(j))+\Delta(\pi)\right)\right|\pi\right]\right)\right]\\
   &\leq\Er \left[\exp\left(\lambda\alpha\left(\sum_{i\in I,j\in I^c}\tilde{d}_{I,\pi(I)}(i,j,\pi(i),\pi(j))+\Delta(\pi)\right)\right)\right]\\
   &\leq \sqrt{\Er[\exp{\left(2\lambda\alpha\Delta(\pi)\right)}]}\sqrt{\Er \left[\exp\left(2\lambda\alpha\sum_{i\in I,j\in I^c}\tilde{d}_{I,\pi(I)}(i,j,\pi(i),\pi(j))\right)\right]}.
\end{aligned}
$$
Since
\[
\Er\left(\sum_{i\ne j} d(i,j,\pi(i),\pi(j))\right) = \frac{1}{N(N-1)}\left(\sum_{i\ne j}\sum_{k\neq \ell} d(i,j,k,\ell)\right) = \frac{1}{N(N-1)}\left(\sum_{i,k\in[N]} d(i,i,k,k)\right),
\]
similarly,
\[
\Er\left(\sum_{i\ne j} d(i,j,\pi(j),\pi(i))\right) = \frac{1}{N(N-1)}\left(\sum_{i\ne j}\sum_{k\neq \ell} d(i,j,k,\ell)\right) = \frac{1}{N(N-1)}\left(\sum_{i,k\in[N]} d(i,i,k,k)\right),
\]
and
\[
\Er\left[\sum_{i=1}^N d(i,i,\pi(i),\pi(i))\right] = \frac{1}{N}\sum_{i,k\in[N]} d(i,i,k,k),
\]
we have,
\[\begin{aligned}
    \Delta(\pi) &= -\frac{(N_1-1)(N-N_1-1)}{N(N-1)(N-2)(N-3)}\left(\sum_{i\ne j} d(i,j,\pi(j),\pi(i))-\Er\left[\sum_{i\ne j} d(i,j,\pi(j),\pi(i))\right]\right)\\
    &\quad+\frac{2(N_1-1)(N-N_1-1)}{N(N-2)(N-3)}\left(\sum_{i=1}^N d(i,i,\pi(i),\pi(i))-\Er\left[\sum_{i=1}^N d(i,i,\pi(i),\pi(i))\right]\right).
\end{aligned}
\]
For simplicity, denote $\Delta_1$ and $\Delta_2$ to be the first and second term of $\Delta$. Since $\Er[\Delta_1]=0$ and \[
\alpha|\Delta_1|\leq \frac{2N(N-1)}{N^2-3N+1}\sqrt{\frac{1}{N(N-1)}\sum_{i\neq j}d(i,j,\pi(i),\pi(j))^2}\leq  \frac{5}{2}\sqrt{\frac{1}{N^2}\sum_{i\neq j}d(i,j,\pi(i),\pi(j))^2}\leq \frac{5}{2}\tilde{V},
\]
by $e^x\leq e^{x^2}+x$, we have
\[
\Er[\exp{\left(4\lambda\alpha\Delta_1(\pi)\right)}]\leq \Er[\exp{\left(16\lambda^2\alpha^2\Delta_1(\pi)^2\right)}]\leq \exp\left(100\lambda^2\tilde{V}\right).
\]
Applying Lemma \ref{lem:1d} to $[\xi(i,k)]$, we have
\[\begin{aligned}
\Er[\exp{\left(4\lambda\alpha\Delta_2(\pi)\right)}]&= \Er\left[\exp{\left(\frac{4\lambda(N-1)}{N^2-3N+1}\sum_{i=1}^N\xi(i,\pi(i))\right)}\right]\\
&\leq \Er\left[\exp{\left(\frac{4\lambda}{N}\sum_{i=1}^N\xi(i,\pi(i))\right)}\right]\\
&\leq\Er\left[\exp{\left(\frac{8\sqrt{3}\lambda}{N}\sum_{i=1}^N\xi(i,\pi(i))G_i\right)}\right]\\
&=\Er\left[\exp{\left(\frac{96\lambda^2}{N^2}\sum_{i=1}^N\xi(i,\pi(i))^2\right)}\right]\\
&\leq\Er\left[\exp{\left(\frac{96\lambda^2}{N^2}\sum_{i,k\in[N]}\xi(i,k)^2\right)}\right]\\
&\leq\Er\left[\exp{\left(\frac{96\lambda^2}{N^2}\sum_{i,k\in[N]}d(i,i,k,k)^2\right)}\right]\\
&\leq \exp\left(96\lambda^2\tilde{V}\right).
\end{aligned}\]
Therefore,
\[
\Er[\exp{\left(2\lambda\alpha\Delta(\pi)\right)}]\leq \sqrt{\Er[\exp{\left(4\lambda\alpha\Delta_1(\pi)\right)}]\Er[\exp{\left(4\lambda\alpha\Delta_2(\pi)\right)}]}\leq \exp\left(98\lambda^2\tilde{V}\right).
\]
Since
\[
\alpha\leq 4+\frac{8}{N-2},
\]
we have
\begin{align}\label{eq:converttopi1pi2}
   &\quad\Er\left[\exp\left(\lambda\left(\sum_{i\ne j} d(i,j,\pi(i),\pi(j))-\Er\left[\sum_{i\ne j} d(i,j,\pi(i),\pi(j))\right]\right)\right)\right]\notag\\
   &\leq \exp\left(50\lambda^2\tilde{V}\right)\sqrt{\Er \left[\exp\left(\lambda\left(8+\frac{16}{N-2}\right)\sum_{i\in I,j\in I^c}\tilde{d}_{I,\pi(I)}(i,j,\pi(i),\pi(j))\right)\right]}\notag\\
   &\leq \exp\left(50\lambda^2\tilde{V}\right)\sqrt{\Er \left[\exp\left(9\lambda\left(\sum_{i\in I,j\in I^c}\tilde{d}_{I,\pi(I)}(i,j,\pi(i),\pi(j))\right)\right)\right]}\notag\\
   &= \exp\left(50\lambda^2\tilde{V}\right)\sqrt{\Er\left[\Er \left[\Er \left[\exp\left(9\lambda\left(\sum_{i\in I,j\in I^c}\tilde{d}_{I,J}(i,j,\pi(i),\pi(j))\right)\right)\middle| \pi(I)=J\right]\middle| I,J\right]\right]}\notag\\
  &= \exp\left(50\lambda^2\tilde{V}\right)\sqrt{\Er\left[\Er \left[\exp\left(9\lambda\left(\sum_{i\in I,j\in I^c}\tilde{d}_{I,J}(i,j,\pi_1(i),\pi_2(j))\right)\right)\middle| I,J\right]\right]}\notag\\
    &= \exp\left(50\lambda^2\tilde{V}\right)\sqrt{\Er\left[\exp\left(9\lambda\left(\sum_{i\in I,j\in I^c}\tilde{d}_{I,J}(i,j,\pi_1(i),\pi_2(j))\right)\right)\right]},
\end{align}
which completes the proof.
\end{proof}

\subsection{Proof of Theorem \ref{thm:randomization}}
\begin{proof}[Proof of Theorem \ref{thm:randomization}]
    By the law of total probability, we may assume that $\D$ is a fixed tensor.     Let $\A_\tau = [a_{\tau}(i,k)]\in \mathbb{R}^{N\times N}$ be a random matrix given by\[
    a_\tau(i,k) = \sum_{j=1}^M d(i,j,k,\tau(j)).
    \]
    Since $d(i,\cdot,k,\cdot)=0$, the entries of $\A_\tau$ sum to 0. Moreover, $\A_{\tau}$ is independent of $\pi$. Then by Lemma \ref{lem:1d},
    \[
\Er\left[\exp\left(\lambda\sum_{i=1}^N a_\tau(i,\pi(i))\right)\right]
\leq \Er\left[\exp\left(2\sqrt{3}\lambda\sum_{i=1}^N a_\tau(i,\pi(i))G_i\right)\right],\
    \]
    or equivalently,
    $$
\Er\left[\exp\left(\lambda\sum_{\substack{i\in[N]\\ j\in[M]}} d(i,j,\pi(i),\tau(j))\right)\right]
\leq \Er\left[\exp\left(2\sqrt{3}\lambda\sum_{\substack{i\in[N]\\ j\in[M]}} d(i,j,\pi(i),\tau(j))G_i\right)\right].
$$
Let $\A_\pi=[a_{\pi}(j,l)]\in \mathbb{R}^{M\times M}$ be a random matrix  independent of $\tau$, given by\[
    a_\pi(j,\ell) = \sum_{i=1}^N d(i,j,\pi(i),\ell)G_i.
    \]
    Since $d(\cdot,j,\cdot,\ell)=0$, the entries of $\A_\pi$ sum to 0 as well. Then by Lemma \ref{lem:1d},
    \[
\Er\left[\exp\left(2\sqrt{3}\lambda\sum_{j=1}^M a_{\pi}(j,\tau(j))\right)\right]
\leq \Er\left[\exp\left(12\lambda\sum_{j=1}^M a_{\pi}(j,\tau(j))G'_j\right)\right],
    \]
    or equivalently,
    $$
\Er\left[\exp\left(2\sqrt{3}\lambda\sum_{\substack{i\in[N]\\ j\in[M]}} d(i,j,\pi(i),\tau(j))G_i\right)\right]
\leq \Er\left[\exp\left(12\lambda\sum_{\substack{i\in[N]\\ j\in[M]}} d(i,j,\pi(i),\tau(j))G_iG'_j\right)\right].
$$
Therefore, we have
\begin{align*}
\Er\left[\exp\left(\lambda\sum_{\substack{i\in[N]\\ j\in[M]}} d(i,j,\pi(i),\tau(j))\right)\right]
&\leq \Er\left[\exp\left(2\sqrt{3}\lambda\sum_{\substack{i\in[N]\\ j\in[M]}} d(i,j,\pi(i),\tau(j))G_i\right)\right]\\
&\leq\Er\left[\exp\left(12\lambda\sum_{\substack{i\in[N]\\ j\in[M]}} d(i,j,\pi(i),\tau(j))G_iG'_j\right)\right],
\end{align*}
which completes the proof.
\end{proof}
\subsection{Proof of Theorem \ref{thm:main}}
\begin{proof}[Proof of Theorem~\ref{thm:main}]
By the Cauchy-Schwarz inequality, we have
\begin{align}
    &\quad\Er\left[\exp\left(\lambda\left({\sum_{i,j\in[N]} d(i,j,\pi(i),\pi(j))}-\Er\left[{\sum_{i,j\in[N]} d(i,j,\pi(i),\pi(j))}\right]\right)\right)\right]\notag\\
    &\leq \sqrt{\Er\left[\exp\left(2\lambda\left({\sum_{i=1}^N d(i,i,\pi(i),\pi(i))}-\Er\left[{\sum_{i=1}^N d(i,i,\pi(i),\pi(i))}\right]\right)\right)\right]}\notag\\
    &\quad \cdot \sqrt{\Er\left[\exp\left(2\lambda\left({\sum_{i\neq j} d(i,j,\pi(i),\pi(j))}-\Er\left[{\sum_{i\neq j} d(i,j,\pi(i),\pi(j))}\right]\right)\right)\right]}. \label{eq:han1}
    \end{align}
For the first term on the right hand side of \eqref{eq:han1}, since
\[
\sum_{i,k\in[N]} \xi(i,k) =0,
\]
by Lemmas \ref{lem:1d} and \ref{lem:1d_MGF}, we have
\begin{align*}
&\quad \Er\left[\exp\left(2\lambda\left({\sum_{i=1}^N d(i,i,\pi(i),\pi(i))}-\Er\left[{\sum_{i=1}^N d(i,i,\pi(i),\pi(i))}\right]\right)\right)\right]\\
&=\Er\left[\exp\left(2\lambda{\sum_{i=1}^N \xi(i,\pi(i))}\right)\right]\\
&\leq \Er\left[\exp\left(4\sqrt{3}\lambda{\sum_{i=1}^N \xi(i,\pi(i))G_i}\right)\right]\\
&\leq\exp\left(\frac{V_1(4\sqrt{3}\lambda)^2-V_1B_1^2(4\sqrt{3}\lambda)^4/4}{2(1-B_1^2(4\sqrt{3}\lambda)^2/2)}\right)\\
&\leq\exp\left(\frac{V_1(4\sqrt{3}\lambda)^2}{2(1-B_1^2(4\sqrt{3}\lambda)^2/2)}\right)\\
&=\exp\left(\frac{48V_1\lambda^2}{2(1-24B_1^2\lambda^2)}\right),
\end{align*}
where $$V_1 = \frac{1}{N}\sum_{i,k\in[N]}\xi(i,k)^2\leq V$$ and $$B_1 = \max_{i,k\in[N]}|\xi(i,k)|\leq \max_{i,k\in[N]}|d(i,i,k,k)|\leq B.$$

For the second term on the right hand side of \eqref{eq:han1}, by Theorems \ref{thm:decoupling} and \ref{thm:randomization},
\begin{align*}
&\Er\left[\exp\left(2\lambda\left({\sum_{i\neq j} d(i,j,\pi(i),\pi(j))}-\Er\left[{\sum_{i\neq j} d(i,j,\pi(i),\pi(j))}\right]\right)\right)\right]\\
\leq& \exp\left(200\lambda^2V_2\right)\sqrt{\Er \left[\exp\left(18\lambda\left(\sum_{i\in I,j\in I^c}\tilde{d}_{I,J}(i,j,\pi_1(i),\pi_2(j))\right)\right)\right]}\\
\leq& \exp\left(200\lambda^2V_2\right)\sqrt{\max_{I,J}\Er \left[\exp\left(216\lambda\left(\sum_{i\in I,j\in I^c}\tilde{d}_{I,J}(i,j,\pi_1(i),\pi_2(j)G_iG'_j)\right)\right)\middle|I,J\right]},
\end{align*}
where $G_i$ and $G'_j$ are defined as in Theorem \ref{thm:randomization} and the maximum is taken over all possible realizations of the random sets $I$ and $J$. 

We now bound the MGF using Lemma \ref{lem:2d_MGF}. To do so, we map the notation of Lemma \ref{lem:2d_MGF} to our current setting by restricting the tensor to $I\times I^c\times J\times J^c$. We set the dimensions  in Lemma \ref{lem:2d_MGF} to $|I|$ and $|I^c|$. Consequently, the independent permutations $\pi$ and $\tau$ in Lemma \ref{lem:2d_MGF} correspond to our independent bijections $\pi_1$ and $\pi_2$. Applying Lemma \ref{lem:2d_MGF} with these dimensions yields \[\begin{aligned}
\max_{I,J}\Er \left[\exp\left(216\lambda\left(\sum_{i\in I,j\in I^c}\tilde{d}_{I,J}(i,j,\pi_1(i),\pi_2(j)G_iG'_j)\right)\right)\middle|I,J\right]\leq \exp\left(\frac{V_2(216\lambda)^2}{2(1-\frac{B_2(216\lambda)}{0.64})}\right),
\end{aligned}\]
where \[
V_2 = \max_{I,J}\frac{1}{N_1(N-N_1)}\sum_{\substack{i\in I,k\in J\\ j\in I^c,\ell\in J^c}}\tilde{d}_{I,J}(i,j,k,\ell)^2
\leq \max_{I,J}\frac{1}{N_1(N-N_1)}\sum_{\substack{i\in I,k\in J\\ j\in I^c,\ell\in J^c}}d_{I,J}(i,j,k,\ell)^2\leq 4.1V\] and \[B_2 = \max_{\substack{I,J\\
\text{bijection }\sigma_1 : I \to J\\\
\text{bijection }\sigma_2 : I^{c} \to J^{c}
}}\norm{\left[\tilde{d}_{I,J}(i,j,\sigma_1(i),\sigma_2(j))\right]_{i,j}}_{\sf op}\leq16\max_{\substack{I,J\\
\text{bijection }\sigma_1 : I \to J\\\
\text{bijection }\sigma_2 : I^{c} \to J^{c}
}}\norm{\left[d_{I,J}(i,j,\sigma_1(i),\sigma_2(j))\right]_{i,j}}_{\sf op}\leq 16B,\]
where the last inequality follows from the fact that taking partial averages cannot increase the maximal operator norm.
Therefore,
\begin{align*}
&\quad\Er\left[\exp\left(\lambda\left({\sum_{i,j\in[N]} d(i,j,\pi(i),\pi(j))}-\Er\left[{\sum_{i,j\in[N]} d(i,j,\pi(i),\pi(j))}\right]\right)\right)\right]\\
&\leq \sqrt{\exp\left(\frac{48V_1\lambda^2}{2(1-24B_1^2\lambda^2)}\right)\exp\left(200V\lambda^2\right)\exp\left(\frac{V_2(216\lambda)^2}{2(1-\frac{B_2(216\lambda)}{0.64})}\right)}\\
&\leq \sqrt{\exp\left(\frac{48V\lambda^2}{2(1-24B^2\lambda^2)}\right)\exp\left(200(4.1V)\lambda^2\right)\exp\left(\frac{4.1V(216\lambda)^2}{2(1-\frac{(16B)(216\lambda)}{0.64})}\right)}\\
&= \exp\!\left(\frac12\left[
      \frac{24\,V\lambda^2}{1-24B^2\lambda^2}
    + 820\,V\lambda^2
    + \frac{4.1 \cdot 216^2\, V\lambda^2}{2(1-\frac{16\cdot 216}{0.64}B\lambda)}
\right]\right) \\
&= \exp\!\left(
\frac{V\lambda^2}{2}\left[
\frac{24}{1-24B^2\lambda^2}
+ 820
+ \frac{191290}{1-5400\,B\lambda}
\right]\right)\\
&\leq \exp\left(\frac{200000V\lambda^2}{2(1-5400B\lambda)}\right)
\end{align*}
for all $\lambda\in[0,1/5400B]$, which implies that, for any $t>0$,
\[
\Prb\left(\sum_{i,j\in[N]} d(i,j,\pi(i),\pi(j))-\Er\left[{\sum_{i,j\in[N]} d(i,j,\pi(i),\pi(j))}\right]\geq t\right)\leq\exp\left(-\frac{t^2}{400000V+10800Bt}\right).
\]
This completes the proof.
\end{proof}

\subsection{Proof of Corollary \ref{cor:uncentered2}}
\begin{proof}[Proof of Corollary~\ref{cor:uncentered2}]
Notice that
\begin{align*}
\sum_{i,j}d_w\big(i,j,\pi(i),\pi(j)\big)
&= \sum_{i,j} w\big(i,j,\pi(i),\pi(j)\big)
 - N\sum_i w(i,\cdot,\cdot,\cdot)
 - N\sum_j w(\cdot,j,\cdot,\cdot) \\
&\quad - N\sum_i w(\cdot,\cdot,\pi(i),\cdot)
 - N\sum_j w(\cdot,\cdot,\cdot,\pi(j)) \\
&\quad + N^2 w(\cdot,\cdot,\cdot,\cdot)+N\sum_i w(i,\cdot,\pi(i),\cdot)
 + N^2 w(\cdot,\cdot,\cdot,\cdot)\\
 &\quad+N^2 w(\cdot,\cdot,\cdot,\cdot)+N\sum_j w(\cdot,j,\cdot,\pi(j)) 
 +N^2 w(\cdot,\cdot,\cdot,\cdot)\\
&\quad - N\sum_i w(i,\cdot,\pi(i),\cdot)
 - N\sum_j w(\cdot,j,\cdot,\pi(j)) \\
&\quad - N\sum_i w(i,\cdot,\pi(i),\cdot)
 - N\sum_j w(\cdot,j,\cdot,\pi(j))
 + N^2 w(\cdot,\cdot,\cdot,\cdot)\\
 &=\sum_{i,j} w\big(i,j,\pi(i),\pi(j)\big)- N\sum_i a_w(i,\pi(i))
 - N^2 w(\cdot,\cdot,\cdot,\cdot),
\end{align*}
where $d_w$ is degenerate and\[\begin{aligned}
a_w(i,j) &=w(i,\cdot,j,\cdot)-w(i,\cdot,\cdot,\cdot)-w(\cdot,\cdot,j,\cdot)+w(\cdot,\cdot,\cdot,\cdot)\\
&+w(\cdot,i,\cdot,j)-w(\cdot,i,\cdot,\cdot)-w(\cdot,\cdot,\cdot,j)+w(\cdot,\cdot,\cdot,\cdot).
\end{aligned}\]
Accordingly, for a general non-centered fourth-order tensor $\W(i,j,k,\ell)\in \mathbb{R}^{N\times N\times N\times N}$, we can decompose the combinatorial sum as \[\begin{aligned}
&\quad\sum_{i,j\in[N]}w(i,j,\pi(i),\pi(j))-\Er\left[\sum_{i,j\in[N]}w(i,j,\pi(i),\pi(j))\right] \\
&= N\left(\sum_{i=1}^N a_w(i,\pi(i))-\Er\left[\sum_{i=1}^N a_w(i,\pi(i))\right]\right)\\
&\quad+\sum_{i,j\in[N]}d_w(i,j,\pi(i),\pi(j))-\Er\left[\sum_{i,j\in[N]}d_w(i,j,\pi(i),\pi(j))\right],
\end{aligned}
\]
where $\mathbf D_w$ is doubly centered and \[
\sum_{i,j\in [N]}a_w(i,j)=0.
\] 
Accordingly, by Theorem~\ref{thm:main} and Sourav Chatterjee’s combinatorial Bernstein inequality \citep{chatterjee2007Stein} (see \citet[Theorem~3.7]{han2024introduction} for the specific formulation adopted here),
\begin{align*}
\Prb\left( Q_w-\Er[ Q_w]>t\right) &\leq \Prb\left(N\sum_{i=1}^N a_w(i,\pi(i))-\Er\left[N\sum_{i=1}^N a_w(i,\pi(i))\right]>\frac{t}{2}\right) \\
&\quad+ \Prb\left(\sum_{i,j\in[N]}d_w(i,j,\pi(i),\pi(j))-\Er\left[\sum_{i,j\in[N]}d_w(i,j,\pi(i),\pi(j))\right]>\frac{t}{2}\right)\\
&\leq \exp\left(-\frac{Kt^2}{N\sum_{i,j\in[N]} a_w(i,j)^2+N\max_{i,j\in[N]}|a_w(i,j)|t}\right) \\
&\quad + \exp\left(-\frac{Kt^2}{V_d+B_dt}\right).
\end{align*}
This completes the proof.
\end{proof}

\subsection{Proof of Theorem \ref{thm:HansonWright}}
\begin{proof}[Proof of Theorem \ref{thm:HansonWright}]
(i) Let $w(i,j,k,\ell) = c_{ij}a_{k\ell }$ for any $i,j,k,\ell\in[N]$ and denote
\[
a_{k\cdot} = \frac{1}{N}\sum_{i=1}^N a_{ki},\quad a_{\cdot\ell} = \frac{1}{N}\sum_{i=1}^N a_{i\ell}\quad\text{and}\quad a_{\cdot\cdot} = \frac{1}{N^2}\sum_{i,j\in[N]} a_{ij}.
\]
Since matrix $\C$ is doubly centered, centering tensor $\mathbf{W}$ gives\[
d_w(i,j,k,\ell) = w(i,j,k,\ell) - w(i,j,\cdot,\ell) - w(i,j,k,\cdot)+ w(i,j,\cdot,\cdot) = c_{ij}\tilde{a}_{k\ell },
\]
where $\tilde{a}_{k\ell } = a_{k\ell } -a_{k\cdot}- a_{\cdot \ell}+a_{\cdot\cdot}$ is the $(k,\ell)$-element of the doubly centered version of matrix $\A$. In addition,
\begin{align*}
    \xi_w(i,k) &= d_w(i,i,k,k)-\frac{1}{N}\sum_{k'=1}^N d_w(i,i,k',k')-\frac{1}{N}\sum_{i'=1}^N d_w(i',i',k,k)+\frac{1}{N^2}\sum_{i',k'\in [N]}d_w(i',i',k',k')\\
    &=\Big(c_{ii}- \frac{1}{N}\sum_{i' = 1}^Nc_{i'i'}\Big)\Big(\tilde a_{kk} - \frac{1}{N}\sum_{k' = 1}^N\tilde a_{k'k'}\Big).
\end{align*}
Using Corollary \ref{cor:uncentered2} on the tensor $w(i,j,k,\ell)$, we obtain
\[
\Prb( Q_{\sf QF}-\Er[ Q_{\sf QF}]\geq t)\leq \exp\left(-\frac{Kt^2}{V_d+B_dt}\right)
.\]
Since 
\[|\tilde{a}_{k\ell}|\leq |a_{k\ell}|+|a_{k\cdot}|+|a_{\cdot \ell}|+|a_{\cdot\cdot}|\leq 4\]
and centering does not increase Frobenius norm in the sense that
\[\norm{\tilde A}_{\sf F}\le \norm{A}_{\sf F}\leq N^2,\]
we have
\begin{align*}
   V_d &= \frac{1}{N}\sum_{i,k\in[N]}\xi_w(i,k)^2+\frac{1}{N^2}\sum_{i,j,k,\ell\in[N]}d_w(i,j,k,\ell)^2\\
   &=\frac{1}{N}\sum_{i,k\in [N]}\Big(c_{ii}- \frac{1}{N}\sum_{i' = 1}^Nc_{i'i'}\Big)^2\Big(\tilde a_{kk} - \frac{1}{N}\sum_{k' = 1}^N\tilde a_{k'k'}\Big)^2+\frac{1}{N^2}\sum_{i,j,k,\ell\in[N]}c_{ij}^2\tilde{a}_{k\ell }^2\\
   &=\frac{1}{N}\sum_{i\in [N]}\Big(c_{ii}- \frac{1}{N}\sum_{i' = 1}^Nc_{i'i'}\Big)^2\sum_{k\in [N]}\Big(\tilde a_{kk} - \frac{1}{N}\sum_{k' = 1}^N\tilde a_{k'k'}\Big)^2+\frac{1}{N^2}\sum_{i,j\in[N]}c_{ij}^2\sum_{k,\ell\in[N]}\tilde{a}_{k\ell }^2\\
   &\leq\frac{1}{N}\sum_{i\in [N]}c_{ii}^2\sum_{k\in [N]}\tilde a_{kk}^2+\frac{1}{N^2}\sum_{i,j\in[N]}c_{ij}^2\sum_{k,\ell\in[N]}\tilde{a}_{k\ell }^2\\
   &\leq\big(\max_k \tilde{a}_{kk}^2+\max_{k,\ell} \tilde{a}_{k\ell}^2\big)\sum_{i,j\in[N]}c_{ij}^2\\
   &= 32\norm{\C}_{\sf F}^2
\end{align*}
and
\[\begin{aligned}
B_d = \max_{\sigma\in\mathcal{S}_N}\norm{\Big[{c}_{ij}\tilde{a}_{\sigma(i)\sigma(j)}\Big]_{i,j}}_{\sf op}\leq \max_{\sigma\in\mathcal{S}_N}\norm{\Big[c_{ij}a_{\sigma(i)\sigma(j)}\Big]_{i,j}}_{\sf op} +\max_{\sigma\in\mathcal{S}_N}\norm{\Big[c_{ij}a_{\sigma(i)\cdot}\Big]_{i,j}}_{\sf op}\\
\quad+ \max_{\sigma\in\mathcal{S}_N}\norm{\Big[c_{ij}a_{\cdot\sigma(j)}\Big]_{i,j}}_{\sf op}+\max_{\sigma\in\mathcal{S}_N}\norm{\Big[c_{ij}a_{\cdot\cdot}\Big]_{i,j}}_{\sf op}.
\end{aligned}\]
Since
\[
\max_{\sigma\in\mathcal{S}_N}
\left\|
\Big[c_{ij} a_{\sigma(i)\cdot}\Big]_{i,j}
\right\|_{\sf op}
\le
\left(\max_i |a_{i\cdot}|\right) \|\C\|_{\sf op}
\leq \norm{\C}_{\sf op},
\]
and similarly 
\[
\max_{\sigma\in\mathcal{S}_N}
\left\|
\Big[c_{ij} a_{\cdot\sigma(j)}\Big]_{i,j}
\right\|_{\sf op}
\le
\left(\max_j |a_{\cdot j}|\right) \|\C\|_{\sf op}
\leq \norm{\C}_{\sf op}
\]
and
\[
\max_{\sigma\in\mathcal{S}_N}
\left\|
\big[c_{ij} a_{\cdot \cdot}\big]_{i,j}
\right\|_{\sf op}
=|a_{\cdot \cdot}|\norm{\C}_{\sf op}
\le \norm{\C}_{\sf op},
\]
we therefore have
\[
B_d\leq 3\norm{\C}_{\sf op} + \max_{\sigma\in\mathcal{S}_N}\norm{\mathbf C\circ \A^{\sigma}}_{\sf op}.
\]
Hence,
\[
\Prb( Q_{\sf QF}-\Er[ Q_{\sf QF}]\geq t)\leq \exp\left(-\frac{Kt^2}{32\norm{\C}_{\sf F}^2+(3\norm{\C}_{\sf op} + \max_{\sigma\in\mathcal{S}_N}\norm{\mathbf C\circ \A^{\sigma}}_{\sf op})t}\right)
,\]
which completes the proof.

(ii) Since matrix $\A\succeq 0$ is positive semidefinite, there exists a Hilbert space $(\mathcal H,\langle\cdot,\cdot\rangle_{\mathcal H})$
with vectors $\bm v_1,\dots,\bm v_N\in\mathcal H$ such that
\[
a_{ij}=\langle \bm v_i,\bm v_j\rangle_{\mathcal H}\qquad (1\le i,j\le N).
\]
We write $\|\bm v\|_{\mathcal H}:=\sqrt{\langle \bm v,\bm v\rangle_{\mathcal H}}$ for the norm induced by the inner product.
In particular,
\[
\|\bm v_i\|_{\mathcal H}^2=\langle \bm v_i,\bm v_i\rangle_{\mathcal H}=a_{ii}\le 1 \qquad (1\le i\le N).
\]

Define the linear map $T:\mathbb R^N\to \mathbb R^N\otimes \mathcal H$ by
\[
T\bm e_i=\bm e_i\otimes \bm v_i,\qquad i=1,\dots,N,
\]
where $\{\bm e_i\}$ is the standard basis in $\mathbb R^N$ and $\otimes$ is the Hilbert-space tensor product. We equip the Hilbert tensor product $\mathbb R^N \otimes \mathcal H$ with the canonical inner product defined on simple tensors by
\[
\langle \bm x_1 \otimes \bm h_1,\ \bm x_2 \otimes \bm h_2\rangle_{\mathbb R^N \otimes \mathcal H}
:= \langle \bm x_1,\bm x_2\rangle_{\mathbb R^N}\,
   \langle \bm h_1,\bm h_2\rangle_{\mathcal H}.
\] Let $T^*:\mathbb R^N\otimes\mathcal H\to \mathbb R^N$ denote the adjoint of $T$, defined by
\[
\langle T \bm x,\ \bm y\rangle_{\mathbb R^N\otimes\mathcal H}=\langle \bm x,\ T^* \bm y\rangle_{\mathbb R^N}
\qquad\forall\,\bm x\in\mathbb R^N,\ \bm y\in\mathbb R^N\otimes\mathcal H.
\]
Let $\mathbf I_{\mathcal H}$ be the identity operator on $\mathcal H$ and
$\C\otimes \mathbf I_{\mathcal H}:\mathbb{R}^N\otimes\mathcal H\to\mathbb{R}^N\otimes\mathcal H$
be the bounded linear map defined by 
$$(\C\otimes \mathbf I_{\mathcal H})(\bm x\otimes \bm h)=(\C\bm x)\otimes \bm h,\qquad \forall\,\bm x\in\mathbb{R}^N,\ \forall\,\bm h\in\mathcal H.$$
Then the composition
$T^*(\C\otimes \mathbf I_{\mathcal H})T:\mathbb R^N\to\mathbb R^N$ is a linear operator and thus can be identified with an $N\times N$ matrix. Under basis $\{\bm e_i\}$, for $1\le i,j\le N$, its $(i,j)$-entry is given by
\[
\bigl(T^*(\C\otimes \mathbf{I}_{\mathcal{H}})T\bigr)_{ij}
= \langle T \bm e_i, (\C\otimes \mathbf{I}_{\mathcal{H}}) T \bm e_j\rangle
= \langle \bm e_i \otimes \bm v_i,\,(\C \bm e_j)\otimes \bm v_j\rangle
= c_{ij}\,\langle \bm v_i,\bm v_j\rangle
= (\C\circ \A)_{ij}.
\]
Define the operator norm of $T$ by
\[
\|T\|_{\mathbb R^N\to \mathbb R^N\otimes \mathcal H}:=\sup_{\|\bm x\|_{\mathbb{R}^N}=1}\|T\bm x\|_{\mathbb{R}^N\otimes\mathcal{H}},
\]
where $\|\cdot\|_{\mathbb{R}^N}$ is the Euclidean norm on $\mathbb R^N$ and
$\|\cdot\|_{\mathbb{R}^N\otimes\mathcal{H}}$ is the norm on
$\mathbb R^N\otimes\mathcal H$ induced by the tensor-product inner product. We may similarly define the operator norms of $T^*$, $\C \otimes \mathbf I_{\mathcal H}$ and $T^*(\C\otimes \mathbf{I}_{\mathcal{H}})T$, as
\[
\|T^*\|_{\mathbb R^N \otimes \mathcal H \to \mathbb R^N},
\quad \|\C \otimes \mathbf I_{\mathcal H}\|_{\mathbb R^N \otimes \mathcal H \to \mathbb R^N \otimes \mathcal H} \quad\text{and}\quad \|T^*(\C\otimes \mathbf{I}_{\mathcal{H}})T\|_{\mathbb R^N \to \mathbb R^N},
\]
respectively. Moreover, since $T^*$ is the adjoint of $T$, we have
\[
\|T^*\|_{\mathbb R^N \otimes \mathcal H \to \mathbb R^N}
=
\|T\|_{\mathbb R^N \to \mathbb R^N \otimes \mathcal H}.
\]
In addition, because $\mathbf I_{\mathcal H}$ is the identity operator, the tensor product operator
$\C \otimes \mathbf I_{\mathcal H}$ satisfies
\[
\|\C \otimes \mathbf I_{\mathcal H}\|_{\mathbb R^N \otimes \mathcal H \to \mathbb R^N \otimes \mathcal H}
=
\|\C\|_{\sf op}.
\] By the above identification of linear operator $T^*(\C\otimes \mathbf{I}_{\mathcal{H}})T$, we have
\[\begin{aligned}
\|\C\circ \A\|_{\sf op} &= \|T^*(\C\otimes \mathbf{I}_{\mathcal{H}})T\|_{\mathbb R^N \to \mathbb R^N}\\
&\leq \|T^*\|_{\mathbb R^N\otimes \mathcal H\to \mathbb R^N} \,\|\C\otimes \mathbf{I}_{\mathcal{H}}\|_{\mathbb R^N \otimes \mathcal H \to \mathbb R^N \otimes \mathcal H}\|T\|_{\mathbb R^N\to \mathbb R^N\otimes \mathcal H}\\
&= \|T\|_{\mathbb R^N\to \mathbb R^N\otimes \mathcal H}^2\,\|\C\|_{\sf op}.
\end{aligned}\]
Since $\|T\|_{\mathbb R^N\to \mathbb R^N\otimes \mathcal H} = \sup_i \|\bm v_i\|_{\mathcal{H}} = \sqrt{\max_i a_{ii}} \le 1$, we have
\[
\|\C\circ \A\|_{\sf{op}} \leq \|\C\|_{\sf{op}}.
\]

If $\A^\sigma = \mathbf{P}_\sigma \A \mathbf{P}_\sigma^\top$ for a permutation matrix $\mathbf{P}_\sigma$, then $\A^\sigma$ is also a positive semidefinite matrix. Repeating the argument yields
\[
\|\C\circ \A^{\sigma}\|_{\sf{op}} \le\|\C\|_{\sf{op}}.
\]
The proof follows from (i).
\end{proof}

\subsection{Proof of Theorem \ref{thm:Bennett}}
\begin{proof}[Proof of Theorem~\ref{thm:Bennett}]
Let $d(i,j,k,\ell)=a_{ij}c_{k\ell }$. By Lemma \ref{lem:decoupling_zero_diagonal}, we have for any convex function $\varphi$,
\begin{align*}
   &\quad\Er\left[\varphi\left(\sum_{i\ne j} d(i,j,\pi(i),\pi(j))\right)\right]\\
   &\leq\frac{1}{2}\,\Er \left[\varphi\left(2\alpha\sum_{i\in I,j\in I^c}\tilde{d}_{I,J}(i,j,\pi_1(i),\pi_2(j))\right)\right]+ \frac{1}{2}\,\Er \left[\varphi\left(2\beta\sum_{i\ne j} d(i,j,\pi(j),\pi(i))\right)\right].
\end{align*}
Since $\alpha\leq 5$ and $\beta\leq 2/N^2$, taking $\varphi(x) = e^{\lambda x/\nu}$ gives
\[
\begin{aligned}
   &\quad\Er\left[\exp\left(\frac{\lambda}{\nu}\sum_{i\ne j} d(i,j,\pi(i),\pi(j))\right)\right]\\
   &\leq\frac{1}{2}\,\Er \left[\exp\left(\frac{2\alpha\lambda}{\nu}\sum_{i\in I,j\in I^c}\tilde{d}_{I,J}(i,j,\pi_1(i),\pi_2(j))\right)\right]+ \frac{1}{2}\,\Er \left[\exp\left(\frac{2\beta\lambda}{\nu}\sum_{i\ne j} d(i,j,\pi(j),\pi(i))\right)\right]\\
   &\leq \frac{1}{2}\,\Er \left[\exp\left(\frac{10\lambda}{\nu}\sum_{i\in I,j\in I^c}\tilde{d}_{I,J}(i,j,\pi_1(i),\pi_2(j))\right)\right]+ \frac{1}{2}\,\Er \left[\exp\left(\frac{4\lambda}{N^2\nu}\sum_{i\ne j} d(i,j,\pi(j),\pi(i))\right)\right].\\
\end{aligned}
\]
We now bound the first term using Lemma \ref{lem:Bennet_2d_decoupled}. To do so, we map the notation of Lemma \ref{lem:Bennet_2d_decoupled} to our current setting by restricting the tensor to $I\times I^c\times J\times J^c$. We set the dimensions  in Lemma \ref{lem:Bennet_2d_decoupled} to $|I|$ and $|I^c|$. Consequently, the independent permutations $\pi$ and $\tau$ in Lemma \ref{lem:Bennet_2d_decoupled} correspond to our independent bijections $\pi_1$ and $\pi_2$. Since restriction and doubly centering will not increase $\norm{\C}_{\sf op}$ and $\norm{\A}_{\sf F}$, and moreover
\[
\max_{\substack{
i \in I,\,k\in J\\
\text{bijection }\sigma_1 : I^{c} \to I^c\\\
\text{bijection }\sigma_2 : I^{c} \to J^{c}
}}\left|\sum_{j\in I^c}\tilde{d}_{I,J}(i,\sigma_1(j),k,\sigma_2(j))\right|\leq\max_{\substack{
i \in I,\, k \in J\\
\text{bijection }\sigma_1 : I^{c} \to I^c\\\
\text{bijection }\sigma_2 : I^{c} \to J^{c}
}}\left|\sum_{j\in I^c}d(i,\sigma_1(j),k,\sigma_2(j))\right|\leq  \nu,
\]we may apply Lemma \ref{lem:Bennet_2d_decoupled} with these dimensions and the above bounds to obtain
$$\Er \left[\exp\left(\frac{10\lambda}{\nu}\sum_{i\in I,j\in I^c}\tilde{d}_{I,J}(i,j,\pi_1(i),\pi_2(j))\right)\right]\leq \exp\left(\frac{5400}{N\nu^2}\norm{\C}_{\sf op}^2\norm{\A}^2_{\sf F}\lambda^2e^{40\lambda}\right).$$
Since
$$
\left|\sum_{i\ne j} d(i,j,\pi(j),\pi(i))\right| = \left|\sum_{i\ne j} c_{ij}a_{\pi(j)\pi(i)}\right| \leq \norm{\C}_{\sf F}\norm{\A}_{\sf F} \leq \sqrt{N}\norm{\C}_{\sf op}\norm{\A}_{\sf F},
$$
and
$$
\Er\left[\frac{4\lambda}{N^2\nu}\sum_{i\ne j} d(i,j,\pi(j),\pi(i))\right] = 0,
$$
by the inequality $e^x\leq x+e^{x^2}$, we have
\[
\Er \left[\exp\left(\frac{4\lambda}{N^2\nu}\sum_{i\ne j} d(i,j,\pi(j),\pi(i))\right)\right]\leq \exp\left(\frac{16}{N^3\nu^2}\norm{\C}_{\sf op}^2\norm{\A}^2_{\sf F}\lambda^2\right). 
\]
Combining the two bounds, we obtain
\[\Er\left[\exp\left(\frac{\lambda}{\nu}\sum_{i\ne j} d(i,j,\pi(i),\pi(j))\right)\right]\leq \exp\left(\frac{5400}{N\nu^2}\norm{\C}_{\sf op}^2\norm{\A}^2_{\sf F}\lambda^2e^{40\lambda}\right).\]
Equivalently,
\[\Er\left[\exp\left(\frac{\lambda}{10\nu}\sum_{i\ne j} d(i,j,\pi(i),\pi(j))\right)\right]\leq \exp\left(\frac{54}{N\nu^2}\norm{\C}_{\sf op}^2\norm{\A}^2_{\sf F}\lambda^2e^{4\lambda}\right).\]
Since $\Er[ Q_{\sf QF}] = 0$, by Lemma \ref{lem:Bennet_MGF}
\[
\Prb\left(\frac{ Q_{\sf QF}-\Er[ Q_{\sf QF}]}{10\nu} \geq t\right)
\leq \exp{\left(-\frac{t}{12}\cdot\log{\left(1+\frac{t}{\frac{54}{N\nu^2}\norm{\C}_{\sf op}^2\norm{\A}_{\sf F}^2}\right)} \right)},
\]
which implies
\[
\Prb\Big( Q_{\sf QF}-\Er[ Q_{\sf QF}] \geq t\Big) 
\leq \exp{\left(-\frac{t}{120\nu}\cdot\log{\left(1+\frac{\nu t}{\frac{540}{N}\norm{\C}_{\sf op}^2\norm{\A}_{\sf F}^2}\right)} \right)}
\]
and thus completes the proof.
\end{proof}

\section{Proofs of auxiliary lemmas}

\subsection{Proof of Lemma \ref{lem:1d}}
\begin{proof}[Proof of Lemma \ref{lem:1d}]
By the law of total probability, we may assume that $\A$ is a fixed and doubly centered matrix and write $\A=[a(i,j)]$. Since $G_i$ is symmetric, we may assume that $\lambda\geq 0$. We construct an exchangeable pair. For any $i,j\in[N]$, let $\pi_{ij}=\pi\circ(i\,j)$ where $(i\,j)$ denotes the transposition when $i\neq j$ and identity otherwise, i.e. for any $k\in [N]$,
\begin{equation}\label{defpiprime}
    \pi_{ij}(k)=
\begin{cases}
\pi(k), & k\neq i,j,\\
\pi(j), & k=i,\\
\pi(i), & k=j.
\end{cases}
\end{equation}
Let $\pi'=\pi_{ij}$, where $i$ and $j$ are sampled independently and uniformly from $[N]$.
Then \((\pi,\pi')\) is exchangeable. Define
\[
F(\pi,\pi')=\frac{N}{2}\!\left(\sum_{i=1}^N a(i,\pi(i))-\sum_{i=1}^N a(i,\pi'(i))\right).
\]
Then \(F(\pi,\pi')=-F(\pi',\pi)\), \(\Er[f(\pi)]=\sum_{i,j}a(i,j)/N=0\) and 
\begin{align*}
    \Er[F(\pi,\pi')\mid\pi]& =\frac{N}{2}\,\Er[a(i,\pi(i))+a(j,\pi(j))-a(i,\pi(j))-a(j,\pi(i))\mid\pi]\\
    &=\sum_{i=1}^N a(i,\pi(i)) - \frac{1}{N}\sum_{i,j\in[N]} a(i,\pi(j))\\
    &=\sum_{i=1}^N a(i,\pi(i))\\
    &=f(\pi),
\end{align*}
where $f$ is introduced in \eqref{eq:1dstatistics}.
Moreover, let
\[
\Delta_{i,j} :=a\bigl(i,\pi(i)\bigr)+a\bigl(j,\pi(j)\bigr)-a\bigl(i,\pi(j)\bigr)-a\bigl(j,\pi(i)\bigr)
\text{ where }i,j\in[N].
\]
Then we have
\[
f(\pi)-f(\pi')=\Delta_{i,j}
\quad \text{and}\quad 
F(\pi,\pi')=\frac{N}{2}\,\Delta_{i,j}.
\]
Hence
\begin{align*}
v(\pi)&=\frac12\,\Er\!\left[\bigl|(f(\pi)-f(\pi'))F(\pi,\pi')\bigr|\,\middle|\,\pi\right]\\
      &=\frac{N}{4}\,\Er\left[\Delta_{i,j}^{2}\middle|\,\pi\right]\\
      &=\frac{1}{4N}\sum_{i,j\in[N]}\bigl(a(i,\pi(i))+a(j,\pi(j))-a(i,\pi(j))-a(j,\pi(i))\bigr)^2.\\
      &\leq \frac{1}{N}\sum_{i,j\in[N]}\left[a(i,\pi(i))^2+a(j,\pi(j))^2+a(i,\pi(j))^2+a(j,\pi(i))^2\right]\\
      &\leq 2\sum_{i=1}^N a(i,\pi(i))^2+ \frac{2}{N}\sum_{i,j\in[N]} a(i,j)^2.
\end{align*}
By Chatterjee's third lemma \citet[Lemma 3.11]{han2024introduction}, taking $\psi = \frac{3}{2}\lambda^2$, we have\[
\log \Er[\exp(\lambda f(\pi))]\leq \frac{\lambda^2r\left(\frac{3}{2}\lambda^2\right)}{2\left(1-\frac{\lambda^2}{\frac{3}{2}\lambda^2}\right)}
= \frac{3}{2}\lambda^2r\left(\frac{3}{2}\lambda^2\right) = \log\Er\left[\exp\left(\frac{3}{2}\lambda^2 v(\pi)\right)\right],
\]
where \(r(\cdot)\) is defined as in \citet[Lemma 3.11]{han2024introduction}; for completeness, we recall that for any \(\psi > 0\),
\[
r(\psi)
=
\frac{1}{\psi}\,\log \Er e^{\psi v(\pi)}.
\]
Since $x\to e^{\frac{3}{2}\lambda^2x}$ is convex, by Jensen's inequality we have
\[\begin{aligned}
  \Er\left[\exp\left(\frac{3}{2}\lambda^2v(\pi)\right)\right]&\leq  \Er\left[\exp\left(\frac{3}{2}\lambda^2\left(2\sum_{i=1}^N a(i,\pi(i))^2+ \frac{2}{N}\sum_{i,j\in[N]} a(i,j)^2\right)\right)\right]\\
  &=  \Er\left[\exp\left(3\lambda^2\left(\sum_{i=1}^N a(i,\pi(i))^2+ \frac{1}{N}\sum_{i,j\in[N]} a(i,j)^2\right)\right)\right]\\
  &=  \Er\left[\exp\left(3\lambda^2\left(\sum_{i=1}^N a(i,\pi(i))^2+ \Er\left[\sum_{i=1}^N a(i,\pi(i))^2\right]\right)\right)\right]\\
  &\leq  \Er\left[\exp\left(6\lambda^2\sum_{i=1}^N a(i,\pi(i))^2\right)\right].\\
\end{aligned}
\]
Since $\bm{G}=(G_1,\ldots,G_N)\sim \mathcal{N}(0,\mathbf{I}_N)$, we have
\[
\sum_{i=1}^N a(i,\pi(i))G_i \sim \mathcal{N}\left(0,\sum_{i=1}^Na(i,\pi(i))^2\right).
\]
Therefore, by the MGF of normal random variable, we obtain
\begin{align*}
\Er\left[\exp\left(\lambda\sum_{i=1}^N a(i,\pi(i))\right)\right]&=\Er[\exp(\lambda f(\pi))]\\&\leq\Er\left[\exp\left(\frac{3}{2}\lambda^2v(\pi)\right)\right]\\
&\leq \Er\left[\exp\left(6\lambda^2\sum_{i=1}^N a(i,\pi(i))^2\right)\right] \\
&= \Er\left[\Er\left[\left.\exp\left(2\sqrt{3}\lambda\sum_{i=1}^N a(i,\pi(i))G_i\right)\right|\pi\right]\right]\\
&= \Er\left[\exp\left(2\sqrt{3}\lambda\sum_{i=1}^N a(i,\pi(i))G_i\right)\right]
\end{align*}
and thus complete the proof.
\end{proof}
\subsection{Proof of Lemma \ref{lem:1d_MGF}}
\begin{proof}[Proof of Lemma \ref{lem:1d_MGF}]
    By the property of normal distribution,
    \[
    \Er\left[\exp\left(\lambda\sum_{i=1}^N a(i,\pi(i))G_i\right)\right] = \exp\left(\frac{\lambda^2}{2}\sum_{i=1}^N a(i,\pi(i))^2\right).
    \]
We construct the same exchangeable pair $(\pi,\pi')$ as in Lemma~\ref{lem:1d}. Define
\[
f_2(\pi)=\sum_{i=1}^N a(i,\pi(i))^2-\frac{1}{N}\sum_{i,j\in[N]}a(i,j)^2, \qquad
F_2(\pi,\pi')=\frac{N}{2}\!\left(\sum_{i=1}^N a(i,\pi(i))-\sum_{i=1}^N a(i,\pi'(i))^2\right).
\]
Then, \(F_2(\pi,\pi')=-F_2(\pi',\pi)\), \(\Er[f_2(\pi)]=0\), and $\Er[F_2(\pi,\pi')\mid\pi]=f_2(\pi)$. Similarly, we have
\[\begin{aligned}
v_2(\pi)&=\frac12\,\Er\!\left[\bigl|(f_2(\pi)-f_2(\pi'))F_2(\pi,\pi')\bigr|\,\middle|\,\pi\right]\\
      &=\frac{1}{4N}\sum_{i,j\in[N]}\bigl(a(i,\pi(i))^2+a(j,\pi(j))^2-a(i,\pi(j))^2-a(j,\pi(i))^2\bigr)^2.\\
      &\leq \frac{1}{2N}\max_{i,j\in[N]}a(i,j)^2\sum_{i,j\in[N]}\left[a(i,\pi(i))^2+a(j,\pi(j))^2+a(i,\pi(j))^2+a(j,\pi(i))^2\right]\\
      &\leq \max_{i,j\in[N]}a(i,j)^2\sum_{i=1}^N a(i,\pi(i))^2+ \frac{1}{N}\max_{i,j\in[N]}a(i,j)^2\sum_{i,j\in[N]}a(i,j)^2\\
      &=\max_{i,j\in[N]}a(i,j)^2f_2(\pi)+ \frac{1}{N}\max_{i,j\in[N]}a(i,j)^2\sum_{i,j\in[N]}a(i,j)^2.
\end{aligned}\]
By Chatterjee's second lemma \citep[Lemma 3.10]{han2024introduction}, we have
\[
\Er\left[\exp\left(\lambda' f_2(\pi)\right)\right]\leq \exp\left(\frac{V_AB_A^2\lambda'^2}{2(1-B_A^2\lambda')}\right),
\]
for all $\lambda'\in [0,1/B_A^2)$. Hence 
\begin{equation}\label{MGFd_nonnegative}
\Er\left[\exp\left(\lambda'\sum_{i=1}^N a(i,\pi(i))^2\right)\right]\leq \exp\left(\frac{V_AB_A^2\lambda'^2}{2(1-B_A^2\lambda')}+ V_A\lambda'\right)=\exp\left(\frac{2V_A\lambda'-V_AB_A^2\lambda'^2}{2(1-B_A^2\lambda')}\right).
\end{equation}
Therefore, for any $|\lambda|\leq \sqrt{2}/B_A$,
\[
\Er\left[\exp\left(\lambda\sum_{i=1}^N a(i,\pi(i))G_i\right)\right] = \exp\left(\frac{\lambda^2}{2}\sum_{i=1}^N a(i,\pi(i))^2\right)\leq \exp\left(\frac{V_A\lambda^2-V_AB_A^2\lambda^4/4}{2(1-B_A^2\lambda^2/2)}\right).
\]
The proof is complete.
\end{proof}

\subsection{Proof of Lemma \ref{lem:2d_MGF}}
\begin{proof}[Proof of Lemma \ref{lem:2d_MGF}]
Applying the logarithmic MGF bound in \citep[Example 2.12]{BoucheronLugosiMassart2013} to the Hermitian dilation of matrix $[d(i,j,\pi(i),\tau(j))]_{i,j}$, one obtains the following MGF bound for decoupled Gaussian bilinear form,
\[
\Er\left[\exp\left(\lambda{\sum_{\substack{i\in[N]\\ j\in[M]}} d(i,j,\pi(i),\tau(j))G_iG'_j}\right)\right]\leq \exp\left(\frac{\norm{[d(i,j,\pi(i),\tau(j))]_{i,j}}_{\sf F}^2\lambda^2}{2(1-\norm{[d(i,j,\pi(i),\tau(j))]_{i,j}}_{\sf op}\lambda)}\right).
\]
Hence,
\[
\begin{aligned}
&\quad\Er\left[\exp\left(\lambda{\sum_{\substack{i\in[N]\\ j\in[M]}} d(i,j,\pi(i),\tau(j))G_iG'_j}\right)\right]\\
&\leq\Er\left.\left[\Er\left[\exp\left(\lambda{\sum_{\substack{i\in[N]\\ j\in[M]}} d(i,j,\pi(i),\tau(j))G_iG'_j}\right)\right|\pi,\tau\right]\right]\\
&\leq\Er\left[\exp\left(\frac{\norm{[d(i,j,\pi(i),\tau(j))]_{i,j}}_{\sf F}^2\lambda^2}{2(1-\norm{[d(i,j,\pi(i),\tau(j))]_{i,j}}_{\sf op}\lambda)}\right)\right]\\
&\leq\Er\left[\exp\left(\frac{\norm{[d(i,j,\pi(i),\tau(j))]_{i,j}}_{\sf F}^2\lambda^2}{2(1-B_{\sf{rect}}\lambda)}\right)\right].
\end{aligned}
\]
It remains to bound the MGF of 
\[
\sum_{\substack{i\in[N]\\ j\in[M]}}d(i,j,\pi(i),\tau(j))^2. 
\]
Fixing $\tau$, let $\mathbf B_\tau=[b_{\tau}(i,k)]\in \mathbb{R}^{N\times N}$ be a tensor given by
\[
    b_\tau(i,k) = \sqrt{\sum_{j=1}^M d(i,j,k,\tau(j))^2}.
    \] By inequality \eqref{MGFd_nonnegative}, we have
\[
\Er\left[\exp\left(\lambda'\sum_{i=1}^N b_\tau(i,\pi(i))^2\right)\right]\leq \exp\left(\frac{2V_\tau\lambda'-V_\tau B_\tau^2\lambda'^2}{2(1-B_\tau^2\lambda')}\right),
\]
where \[
B_\tau = \max_{i,k\in[N]}|b_{\tau}(i,k)|=\sqrt{\max_{i,k\in[N]}\sum_{j=1}^M d(i,j,k,\tau(j))^2}~~{\rm and}~~ 
 V_\tau = \frac{1}{N}\sum_{i,k\in[N]}b_\tau(i,k)^2=\frac{1}{N}\sum_{\substack{i,k\in[N]\\ j\in[M]}}d(i,j,k,\tau(j))^2.
 \]
Since
\[\begin{aligned}
B_\tau&=\sqrt{\max_{i,k\in[N]}\sum_{j=1}^{M}d(i,j,k,\tau(j))^2}=
\sqrt{\max_{\substack{i\in[N]\\ \sigma\in\mathcal{S}_N}}\sum_{j=1}^{M}d(i,j,\sigma(i),\tau(j))^2}\leq \max_{\substack{\sigma\in\mathcal{S}_N\\ \tilde{\sigma}\in\mathcal{S}_M}}\norm{\left[d(i,j,\sigma(i),\tilde{\sigma}(j))\right]_{i,j}}_{\sf op}=B_{\sf{rect}},
  \end{aligned}
\]
we have, for any $|\lambda'|<1/B_{\sf{rect}}^2$,
\[
\Er\left[\exp\left(\lambda'\sum_{\substack{i\in[N]\\ j\in[M]}}d(i,j,\pi(i),\tau(j))^2\right)\right]=
\Er\left[\exp\left(\lambda'\sum_{i=1}^N b_\tau(i,\pi(i))^2\right)\right]\leq \exp\left(\frac{V_\tau\left(2\lambda'- B_{\sf{rect}}^2\lambda'^2\right)}{2\left(1-B_{\sf{rect}}^2\lambda'\right)}\right).
\]
It remains to bound the MGF for $V_\tau$. Let $\mathbf B_\Sigma = [b_{\Sigma}(j,l)]\in \mathbb{R}^{M\times M}$ be a fixed matrix given by\[
    b_\Sigma(j,\ell) = \sqrt{\frac{1}{N}\sum_{i,k\in[N]}d (i,j,k,\ell)^2}.
    \]
Then by inequality \eqref{MGFd_nonnegative}, we have
    \[
\exp\left(\lambda''\sum_{j=1}^M b_\Sigma(j,\tau(j))^2\right)\leq\exp\left(\frac{2V_\Sigma \lambda''-V_\Sigma B_\Sigma^2\lambda''^2}{2\left(1-B_\Sigma^2\lambda''\right)}\right),
\]
where
\[B_\Sigma = \max_{j,\ell\in[M]}|b_\Sigma(j,\ell)| = \sqrt{\max_{j,\ell\in[M]}\frac{1}{N}\sum_{i,k\in[N]}b_\Sigma (i,j,k,\ell)^2}
    \] and 
\[V_\Sigma = \frac{1}{M}\sum_{j,l\in[M]}b_\Sigma(i,j)^2 = \frac{1}{NM}\sum_{\substack{i,k\in[N]\\ j,\ell\in[M]}}d(i,j,k,\ell)^2=\tilde V_{\sf{rect}}.\]
Since
\begin{align*}
  B_\Sigma&=\sqrt{\max_{j,\ell\in[M]}\frac{1}{N}\sum_{i,k\in[N]}d(i,j,k,\ell)^2} = \sqrt{\Er\left[\max_{j,\ell\in[M]}\sum_{i=1}^Nd(i,j,\pi(i),\ell)^2\right]}=\sqrt{\max_{\substack{j,\ell\in[M]\\ \sigma\in\mathcal{S}_N}}\sum_{i=1}^Nd(i,j,\sigma(i),\ell)^2}\\
  &=\sqrt{\max_{\substack{j\in [M]\\ \sigma\in\mathcal{S}_N\\ \tilde{\sigma}\in\mathcal{S}_M}}\sum_{i=1}^Nd(i,j,\sigma(i),\tilde{\sigma}(j))^2}\leq \max_{\substack{\sigma\in\mathcal{S}_N\\ \tilde{\sigma}\in\mathcal{S}_M}}\norm{\left[d(i,j,\sigma(i),\tilde{\sigma}(j))\right]_{i,j}}_{\sf op}=B_{\sf{rect}},
\end{align*}
we have
\[
\Er\left[\exp\left(\lambda''\sum_{j=1}^M b_\Sigma(j,\tau(j))^2\right)\right]\leq\exp\left(\frac{2\tilde V_{\sf{rect}} \lambda''-\tilde V_{\sf{rect}} B_{\sf{rect}}^2\lambda''^2}{2\left(1-B_{\sf{rect}}^2\lambda''\right)}\right).
\]
Letting  $\lambda'\in[0,(2-\sqrt{2})/B_{\sf{rect}}^2)$, then taking $\lambda'' = \frac{2\lambda'- B_{\sf{rect}}^2\lambda'^2}{2\left(1-B_{\sf{rect}}^2\lambda'\right)}\in [0,1/B_{\sf{rect}}^2)$, we have 
\[
\Er\left[\exp\left(\lambda'\sum_{\substack{i\in[N]\\ j\in[M]}}d(i,j,\pi(i),\tau(j))^2\right)\right]\leq \exp\left(
\frac{
\tilde V_{\sf{rect}}\lambda'(B_{\sf{rect}}^2\lambda' - 2)(B_{\sf{rect}}^{4}\lambda'^2 - 6B_{\sf{rect}}^2\lambda' + 4)
}{
4(B_{\sf{rect}}^2\lambda' - 1)(B_{\sf{rect}}^4\lambda'^2 - 4B_{\sf{rect}}^2\lambda' + 2)
}
\right).
\]
Since
\[
\exp\left(
\frac{
\tilde V_{\sf{rect}}\lambda'(B_{\sf{rect}}^2\lambda' - 2)(B_{\sf{rect}}^{4}\lambda'^2 - 6B_{\sf{rect}}^2\lambda' + 4)
}{
4(B_{\sf{rect}}^2\lambda' - 1)(B_{\sf{rect}}^4\lambda'^2 - 4B_{\sf{rect}}^2\lambda' + 2)
}\right)\leq \exp\left(\frac{\tilde V_{\sf{rect}}\lambda'}{1-\frac{B_{\sf{rect}}^2\lambda'}{2-\sqrt{2}}}\right),
\]
we have
\[
\Er\left[\exp\left(\lambda'\sum_{\substack{i\in[N]\\ j\in[M]}}d(i,j,\pi(i),\tau(j))^2\right)\right]\leq\exp\left(\frac{\tilde V_{\sf{rect}}\lambda'}{1-\frac{B_{\sf{rect}}^2\lambda'}{2-\sqrt{2}}}\right).
\]
Let $\lambda\in[0,0.64/B_{\sf{rect}})$, and take $\lambda' = \frac{\lambda^2}{2(1-B_{\sf{rect}}\lambda)}\in[0,(2-\sqrt
2)/B_{\sf{rect}}^2)$. Then,
\[
\exp\left(\frac{\tilde V_{\mathrm{rect}}\,\lambda'}{1-\frac{B_{\mathrm{rect}}^{2}\lambda'}{2-\sqrt2}}\right)
=
\exp\left(\frac{\tilde V_{\mathrm{rect}}\,\lambda^{2}}{2\Bigl(1-B_{\sf{rect}}\lambda-\frac{B_{\sf{rect}}^2\lambda^{2}}{2(2-\sqrt2)}\Bigr)}\right).
\]
Moreover, for $x\in[0,0.64]$, we can check that
\[
1-x-\frac{x^{2}}{2(2-\sqrt2)}\ge 1-\frac{x}{0.64}.
\]
Hence, we have
\[
\Er\left[\exp\left(\frac{\norm{[d(i,j,\pi(i),\tau(j))]_{i,j}}_{\sf F}^2\lambda^2}{2(1-B_{\sf{rect}}\lambda)}\right)\right]\leq\exp\left(\frac{\tilde V_{\sf{rect}}\lambda^2}{2(1-\frac{B_{\sf{rect}}\lambda}{0.64})}\right).
\]
Therefore, we obtain
\[\Er\left[\exp\left(\lambda{\sum_{\substack{i\in[N]\\ j\in[M]}} d(i,j,\pi(i),\tau(j))G_iG'_j}\right)\right]\leq\exp\left(\frac{\tilde V_{\sf{rect}}\lambda^2}{2(1-\frac{B_{\sf{rect}}\lambda}{0.64})}\right)\]
and complete the proof.
\end{proof}

\subsection{Proof of Lemma \ref{lem:Bennet_MGF}}
\begin{proof}[Proof of Lemma \ref{lem:Bennet_MGF}]
By Chernoff's method, for any \(\lambda\ge 0\),
\[
\Prb\big(X\ge t\big)
\le
\exp(-\lambda t)\Er e^{\lambda X}
\le
\exp\big(-\lambda t+C\lambda^{2}e^{4\lambda}\big).
\]
Choose \(\lambda\) such that
\[
(4\lambda^2+2\lambda)e^{4\lambda} = \frac{t}{C}. 
\]
For any $x>0$, we can show that
\[
\left(4\left(\frac{1}{6}\log(1+x)\right)^2+2\left(\frac{1}{6}\log(1+x)\right)\right)e^{4\left(\frac{1}{6}\log(1+x)\right)}\leq x.
\]
Then, by monotonicity of $(4\lambda^2+2\lambda)e^{4\lambda}$, we have
\[
\lambda \leq \frac{1}{6}\log\left(1+\frac{t}{C}\right).
\]
Consequently,
\[
-\lambda t+C\lambda^{2}e^{4\lambda} = -\lambda t+\frac{\lambda^2}{4\lambda^2+2\lambda}t=-\frac{4\lambda^3+\lambda^2}{4\lambda^2+2\lambda}t\leq -\frac{\lambda}{2}t\leq -\frac{1}{12}t\log\left(1+\frac{t}{C}\right),
\]
and the proof is thus complete.
\end{proof}
\subsection{Proof of Lemma \ref{lem:Bennet_1d_lem1}}
\begin{proof}[Proof of Lemma \ref{lem:Bennet_1d_lem1}]
This proof follows that of \citet[Theorem 3.3]{polaczyk2023concentrationboundssamplingreplacement}. We include it here for completeness. For any $i,j$, denote $f_{ij} = \sum_{k=1}^N a(k,\pi_{ij}(k))$ where $\pi_{ij}$ is defined as in \eqref{defpiprime}. Since $a(i,j)\in[0,1]$, we have
\[\begin{aligned}
\quad\sum_{i,j\in[N]} (f_{ij}-f)_{+}&
= \sum_{i,j\in[N]} (a(i,\pi(j))+a(j,\pi(i))-a(i,\pi(i))-a(j,\pi(j)))_{+}\\
&\le \sum_{i,j\in[N]} (a(i,\pi(j))+a(j,\pi(i)))\\
&= 2\sum_{i,j\in [N]} a(i,j)\\
&= 2N\,\Er[f],
\end{aligned}\]
where $x_+:=\max\{x,0\}$ for $x\in \mathbb R$.
For all positive functions $f$, define its entropy functional by $\operatorname{Ent}(f)
:= \Er[ f \log f]-\Er[f]\log\Er[f]$. By the modified log-Sobolev inequality and the convexity of
$x\mapsto e^{2x}$, we obtain
\begin{align*}
\operatorname{Ent}(e^{\lambda f})
&\le \frac{\lambda}{N}\,
\Er\left[ e^{\lambda f}
\sum_{i,j\in[N]} \bigl(e^{\lambda (f_{ij}-f)_{+}}-1\bigr)(f_{ij}-f)_{+}\right] \\
&\le \frac{\lambda}{N}(e^{2\lambda}-1)\,
\Er\left[ e^{\lambda f}
\sum_{i,j\in[N]} (f_{ij}-f)_{+}\right] \\
&\le 2\lambda (e^{2\lambda}-1)\,
\Er [f] \,\Er [e^{\lambda f}] \\
&\le 4\lambda^{2} e^{2\lambda}\,
\Er [f] \,\Er [e^{\lambda f}],
\end{align*}
for all $\lambda\ge 0$. Hence, using \citet[Proposition C.1]{polaczyk2023concentrationboundssamplingreplacement}, with
$a = 4\,\Er[f]$ and $b = 2$ gives the conclusion.
\end{proof}
\subsection{Proof of Lemma \ref{lem:Bennet_1d_lem2}}
\begin{proof}[Proof of Lemma \ref{lem:Bennet_1d_lem2}]
Denote $\tilde{a}(i,j)$ as the $(i,j)$-entry of the doubly centered version $\bf\tilde{A}$ of matrix $\bf A$, given by
\[
\tilde{a}(i,j) = a(i,j) -\frac{1}{N}\sum_{k=1}^N a(k,j)- \frac{1}{N}\sum_{\ell=1}^N a(i,\ell) - \frac{1}{N}\sum_{k,\ell\in [N]} a(k,\ell).
\]
Thus,
\[
\sum_{i=1}^N \tilde{a}(i,\pi(i)) = f-\Er[f].
\]
By Lemma \ref{lem:1d} and Lemma \ref{lem:1d_MGF}, for any $|\lambda|\leq \frac{\sqrt{2}}{2\sqrt{3}}$, we have
\begin{align*}
\Er[\exp(\lambda(f-\Er[f]))]&\leq \Er\left[\exp\left(2\sqrt{3}\lambda\sum_{i=1}^N \tilde{a}(i,\pi(i))G_i\right)\right]\\
&= \Er\left[\exp\left(6\lambda^2\sum_{i=1}^N \tilde{a}(i,\pi(i))^2\right)\right]\\
&\leq \Er\left[\exp\left(6\lambda^2\sum_{i=1}^N a(i,\pi(i))^2\right)\right]\\
&=\Er\left[\exp\left(2\sqrt{3}\lambda\sum_{i=1}^N a(i,\pi(i))G_i\right)\right]\\
&\leq \exp\left(\frac{V_A(2\sqrt{3}\lambda)^2(1-(2\sqrt{3}\lambda)^2/4)}{2(1-(2\sqrt{3}\lambda)^2/2)}\right).
\end{align*}
We conclude the proof since 
\[
\frac{V_A(2\sqrt{3}\lambda)^2(1-(2\sqrt{3}\lambda)^2/4)}{2(1-(2\sqrt{3}\lambda)^2/2)} = \frac{6(1-3\lambda^2)}{(1-6\lambda^2)}V_A\lambda^2\leq 12V_A\lambda^2
\]
for any $|\lambda|\leq 1/3$.
\end{proof}

\subsection{Proof of Lemma \ref{lem:Bennet_1d}}
\begin{proof}[Proof of Lemma \ref{lem:Bennet_1d}]
Noting that double centering reduces $V_A$ while leaving $f-\Er[f]$ invariant, we assume without loss of generality that $\A$ is a doubly centered matrix. Similar to the proof of \citet[Theorem 3.1]{polaczyk2023concentrationboundssamplingreplacement}, for a fixed $\lambda>0$, set
$$
\rho = \frac{1}{3\lambda}.
$$
Denote 
\[
f^{\downarrow} := \sum_{i=1}^N a(i,\pi(i)) \cdot\ind(|a(i,\pi(i))| \leq \rho)
~~ {\rm and} ~~
f^{\uparrow} := \sum_{i=1}^N |a(i,\pi(i))| \cdot\ind(|a(i,\pi(i))| > \rho),
\]
so that $f \leq f^{\downarrow} + f^{\uparrow}$. We then aim to control the MGFs for $f^{\downarrow}$ and $f^{\uparrow}$. By Lemma \ref{lem:Bennet_1d_lem2} applied to $f^{\downarrow}/\rho$, we have for any $\lambda'\in\left[0,\frac{1}{3}\right]$,
$$
\Er[\exp(\lambda' (f^{\downarrow}/\rho-\Er[f^{\downarrow}/\rho]))]\leq \exp\left(\frac{12V_A\lambda'^2}{\rho^2}\right),
$$
which implies
$$
\Er[\exp(3\lambda\lambda' (f^{\downarrow}-\Er[f^{\downarrow}]))]\leq \exp\left(12V_A(3\lambda\lambda')^2\right).
$$
Taking $\lambda' = \frac{1}{3}$, we get
$$
\Er[\exp(\lambda (f^{\downarrow}-\Er[f^{\downarrow}]))]\leq \exp\left(12V_A\lambda^2\right).
$$
By the definitions of $f^{\uparrow}$ and $\rho$ ,
\[
\Er [f^{\uparrow}] \leq \frac{V_A}{\rho} =3V_A\lambda.
\]
Hence by Lemma \ref{lem:Bennet_1d_lem1} applied to $f^{\uparrow}$,
$$
\Er[\exp(\lambda (f^{\uparrow}-\Er[f^{\uparrow}]))]\leq \exp(2\Er[f^{\uparrow}]\lambda (e^{2\lambda}-1))\leq \exp\left(6V_A\lambda^2(e^{2\lambda}-1)\right).
$$
Using the assumption $\Er [f] = 0$ and triangle inequality, we obtain
$$
|\Er [f^{\downarrow}]| = |\Er [f^{\downarrow}] - \Er [f]| \leq \Er [f^{\uparrow}] \leq 3V_A\lambda.
$$
Combining the two MGF bounds we arrive at, for any $\lambda\geq 0$,
\begin{align*}
\Er\left[\exp\left(\lambda f\right)\right]
&\leq \Er\left[\exp\left(\lambda \left(f^{\downarrow} + f^{\uparrow}\right)\right)\right] \\
&\leq \frac{1}{2}\left(\Er\left[\exp\left(2\lambda f^{\downarrow}\right)\right]+\Er\left[\exp\left(2\lambda f^{\uparrow}\right)\right]\right) \\
&\leq \frac{1}{2}\exp\left(6V_A\lambda^2\right)\left(\Er\left[\exp\left(2\lambda \left(f^{\downarrow}-\Er\left[f^{\downarrow}\right]\right)\right)\right]+\Er\left[\exp\left(2\lambda \left(f^{\uparrow}-\Er\left[f^{\uparrow}\right]\right)\right)\right]\right) \\
&\leq \frac{1}{2}\exp\left(6V_A\lambda^2\right)\left(\exp\left(48V_A\lambda^2\right)+\exp\left(24V_A\lambda^2\left(e^{4\lambda}-1\right)\right)\right)\\
&\leq \exp\left(54V_A\lambda^2e^{4\lambda}\right).
\end{align*}
Consequently, by Lemma  \ref{lem:Bennet_MGF},  we have
$$
\Prb(f\geq t)
\le
\exp\left(-\frac{t}{12}\log\left(1+\frac{t}{54V_A}\right)\right)
$$
and the proof is thus complete.
\end{proof}

\subsection{Proof of Lemma \ref{lem:Bennet_2d_decoupled}}
\begin{proof}[Proof of Lemma \ref{lem:Bennet_2d_decoupled}]
    Given $\tau$, let $b_{ik} = \frac{1}{\nu_{\sf rect}} \sum_{j=1}^M c_{ij}a_{k\tau(j)}$, and thus we have $|b_{ik}|\leq 1$ and
    $$
    \frac{ Q_{\sf BF}}{\nu_{\sf rect}} = \frac{1}{\nu_{\sf rect}}\sum_{\substack{i\in [N]\\j\in[M]}} c_{i j} a_{\pi(i)\tau(j)} = \sum_{i=1}^Nb_{i\pi(i)}.
    $$
    By Lemma \ref{lem:Bennet_1d} applied to $b_{ik}$, we have
    $$
    \Er[\exp{(\lambda ( Q_{\sf BF}-\Er[ Q_{\sf BF}\mid\tau])/\nu_{\sf rect})}\mid\tau]
\leq \exp(54\Er[\Sigma_\tau^2]\lambda^2e^{4\lambda}),
    $$
where
    $$   
\Er[\Sigma_\tau^2] = \frac{1}{N}\sum_{\substack{i,k\in [N]}} b_{ik}^2 = \frac{1}{N\nu_{\sf rect}^2}\sum_{\substack{i,k\in [N]}}\left(\sum_{j=1}^M c_{ij}a_{k\tau(j)}\right)^2\leq \frac{1}{N\nu_{\sf rect}^2}\norm{\C}_{\sf op}^2\norm{\A}^2_{\sf F},
    $$
for all $\tau \in \mathcal{S}_M$. Therefore,
    $$
    \Er[\exp{(\lambda ( Q_{\sf BF}-\Er[ Q_{\sf BF}\mid\tau])/\nu_{\sf rect})}]
\leq \exp\left(\frac{54}{N\nu_{\sf rect}^2}\norm{\C}_{\sf op}^2\norm{\A}^2_{\sf F}\lambda^2e^{4\lambda}\right).
    $$
Since matrix $\C$ is doubly centered, we have
\[
\Er[ Q_{\sf BF}\mid\tau] = \frac{1}{N}\sum_{\substack{i,k\in[N]\\ j\in[M]}}c_{ij}a_{k\tau(j)} = \frac{1}{N}
\sum_{\substack{i\in[N]\\ j\in[M]}}c_{ij}\sum_{k=1}^N a_{k\tau(j)}=0.
\]
Therefore, $\Er[ Q_{\sf BF}]=0=\Er[ Q_{\sf BF}\mid\tau]$, which implies
$$
    \Er[\exp{(\lambda ( Q_{\sf BF}-\Er[ Q_{\sf BF}])/\nu_{\sf rect})}]
\leq \exp\left(\frac{54}{N\nu_{\sf rect}^2}\norm{\C}_{\sf op}^2\norm{\A}^2_{\sf F}\lambda^2e^{4\lambda}\right).
    $$
    Hence, by Lemma \ref{lem:Bennet_MGF}, for any $t\geq0$,
\[
\Prb\Big( Q_{\sf BF}-\Er[ Q_{\sf BF}] \geq t\Big)
\leq \exp{\left(-\frac{t}{12\nu_{\sf rect}}\cdot\log{\left(1+\frac{\nu_{\sf rect} t}{\frac{54}{N}\norm{\C}_{\sf op}^2\norm{\A}_{\sf F}^2}\right)} \right)},
\]
and the proof is then complete.
\end{proof}

\subsection{Proof of Lemma \ref{lem:decoupling_zero_diagonal}}

\begin{proof}[Proof of Lemma~\ref{lem:decoupling_zero_diagonal}]
    By the assumptions on matrices $\C$ and $\A$, we have, for all $i,j,k,\ell\in[N]$, 
    \[
    d(i,j,k,\cdot) = d(i,j,\cdot,\ell) = d(i,j,k,k)=0.
    \]
Conditioning on $\pi$, we have $$\begin{aligned}
        &\quad\Er\left[\left.\sum_{i\in I,j\in I^c}\tilde{d}_{I,\pi(I)}(i,j,\pi(i),\pi(j))\right|\pi\right]\\
        &=\Er\left[\left.\sum_{i\in I,j\in I^c}d(i,j,\pi(i),\pi(j))\right|\pi\right]-\Er\left[\left.\sum_{i\in I,j\in I^c}\frac{1}{|I|}\sum_{k\in \pi(I)} d(i,j,k,\pi(j))\right|\pi\right]\\
        &\quad-\Er\left[\left.\sum_{i\in I,j\in I^c}\frac{1}{|I^{c}|}\sum_{\ell\in \pi(I^{c})} d(i,j,\pi(i),\ell)\right|\pi\right]+\Er\left[\left.\sum_{i\in I,j\in I^c}\frac{1}{|I||I^c|}\sum_{k\in \pi(I),\ell\in\pi(I^c)} d(i,j,k,\ell)\right|\pi\right]. 
    \end{aligned}
$$
We have
$$
\Er\left[\left.\sum_{i\in I,j\in I^c}d(i,j,\pi(i),\pi(j))\right|\pi\right] = \frac{\binom{N-2}{N_1-1}}{\binom{N}{N_1}}\sum_{i\neq j}d(i,j,\pi(i),\pi(j)) = \frac{N_1(N-N_1)}{N(N-1)}\sum_{i\neq j}d(i,j,\pi(i),\pi(j)) 
$$
and
\[
\begin{aligned}
&\quad\Er\left[\sum_{i\in I, j\in I^{c}}\frac{1}{|I|}\sum_{k\in \pi(I)} d(i,j,k,\pi(j))\middle|\pi\right] \\
&= \frac{1}{N_1}\sum_{i\neq j}\sum_{k=1}^{N}
d(i,j,k,\pi(j))\,
\Prb\left(i\in I, j\notin I, \pi^{-1}(k)\in I\right) \\
&= \frac{1}{N_1}\sum_{i\ne j}\Bigg[
\sum_{\pi^{-1}(k)\notin\{i,j\},k}
d(i,j,k,\pi(j))\frac{\binom{N-3}{N_1-2}}{\binom{N}{N_1}}
+
d\big(i,j,\pi(i),\pi(j)\big)\frac{\binom{N-2}{N_1-1}}{\binom{N}{N_1}}
+
d\big(i,j,\pi(j),\pi(j)\big)\cdot 0
\Bigg] \\
&=
\frac{\binom{N-3}{\,N_1-2\,}}{N_1\binom{N}{\,N_1\,}}
\sum_{i\ne j}\sum_{k=1}^{N} d(i,j,k,\pi(j))
+\left(
\frac{\binom{N-2}{\,N_1-1\,}}{N_1\binom{N}{\,N_1\,}}
-\frac{\binom{N-3}{\,N_1-2\,}}{N_1\binom{N}{\,N_1\,}}
\right)\sum_{i\ne j} d(i,j,\pi(i),\pi(j))
\\
&\quad-\frac{\binom{N-3}{\,N_1-2\,}}{N_1\binom{N}{\,N_1\,}}
\sum_{i\ne j} d(i,j,\pi(j),\pi(j)) \\
&=\frac{(N-N_1)(N-N_1-1)}{N(N-1)(N-2)} \sum_{i\ne j} d(i,j,\pi(i),\pi(j))
-
\frac{(N_1-1)(N-N_1)}{N(N-1)(N-2)} \sum_{i\ne j} d(i,j,\pi(j),\pi(j))\\
&=\frac{(N-N_1)(N-N_1-1)}{N(N-1)(N-2)} \sum_{i\ne j} d(i,j,\pi(i),\pi(j)).
\end{aligned}
\]
Similarly,
\begin{align*}
&\quad\Er\left[\sum_{i\in I, j\in I^{c}}\frac{1}{|I^c|}\sum_{\ell\in \pi(I^c)} d(i,j,\pi(i),\ell)\middle|\pi\right] \\
&=\frac{N_1(N_1-1)}{N(N-1)(N-2)} \sum_{i\ne j} d(i,j,\pi(i),\pi(j))
-
\frac{(N-N_1-1)N_1}{N(N-1)(N-2)} \sum_{i\ne j} d(i,j,\pi(i),\pi(i))\\
&=\frac{N_1(N_1-1)}{N(N-1)(N-2)} \sum_{i\ne j} d(i,j,\pi(i),\pi(j)).
\end{align*}
We have 
\begin{align*}
&\quad\Er\left[\sum_{i\in I, j\in I^{c}}\frac{1}{|I||I^{c}|}
\sum_{k\in \pi(I), \ell\in \pi(I^{c})} d(i,j,k,\ell)\middle|\pi\right] \\[2mm]
&= \frac{1}{N_1(N-N_1)}\sum_{i\ne j}\sum_{k\ne \ell}
d(i,j,k,\ell)\Prb\left(i\in I, j\notin I, \pi^{-1}(k)\in I, \pi^{-1}(\ell)\notin I\right) \\
&= \frac{1}{N_1(N-N_1)}\sum_{i\ne j}\Bigg[
\frac{\binom{N-4}{N_1-2}}{\binom{N}{N_1}}
\sum_{\substack{k\notin\{\pi(i),\pi(j)\}\\ \ell\notin\{\pi(i),\pi(j)\}, \ell\ne k}}
d(i,j,k,\ell)
+
\frac{\binom{N-3}{N_1-1}}{\binom{N}{N_1}}
\sum_{\substack{\ell\notin\{\pi(i),\pi(j)\}}}
d(i,j,\pi(i),\ell) \\[1mm]
&\quad+
\frac{\binom{N-3}{N_1-2}}{\binom{N}{N_1}}
\sum_{\substack{k\notin\{\pi(i),\pi(j)\}}}
d(i,j,k,\pi(j))\quad+
\frac{\binom{N-2}{N_1-1}}{\binom{N}{N_1}}
d(i,j,\pi(i),\pi(j))
\Bigg] \\
&= \frac{1}{N_1(N-N_1)}\sum_{i\ne j}\Bigg[
\frac{N_1(N_1-1)(N-N_1)(N-N_1-1)}{N(N-1)(N-2)(N-3)}
\sum_{\substack{k\notin\{\pi(i),\pi(j)\}\\ \ell\notin\{\pi(i),\pi(j)\}, \ell\ne k}}
d(i,j,k,\ell) \\
&\quad+\frac{N_1(N-N_1)(N-N_1-1)}{N(N-1)(N-2)}
\sum_{\substack{\ell\notin\{\pi(i),\pi(j)\}}}
d(i,j,\pi(i),\ell)
+
\frac{N_1(N_1-1)(N-N_1)}{N(N-1)(N-2)}
\sum_{\substack{k\notin\{\pi(i),\pi(j)\}}}
d(i,j,k,\pi(j)) \\[1mm]
&\quad+\frac{N_1(N-N_1)}{N(N-1)}d(i,j,\pi(i),\pi(j))
\Bigg]\\
&= -\frac{1}{N_1(N-N_1)}\sum_{i\ne j}\Bigg[
\frac{N_1(N_1-1)(N-N_1)(N-N_1-1)}{N(N-1)(N-2)(N-3)}\big(-d(i,j,\pi(i),\pi(j))-d(i,j,\pi(j),\pi(i))\\
&\quad- 2d(i,j,\pi(i),\pi(i))-2d(i,j,\pi(j),\pi(j))+\sum_k d(i,j,k,k)\big)\\
&\quad+\frac{N_1(N-N_1)(N-N_1-1)}{N(N-1)(N-2)}
\big(d(i,j,\pi(i),\pi(i))+d(i,j,\pi(i),\pi(j))\big)\\
&\quad+
\frac{N_1(N_1-1)(N-N_1)}{N(N-1)(N-2)}
(d(i,j,\pi(i),\pi(j))+d(i,j,\pi(j),\pi(j))) \\
&\quad-\frac{N_1(N-N_1)}{N(N-1)}d(i,j,\pi(i),\pi(j))
\Bigg]\\
&=\frac{(N_1-1)(N-N_1-1)}{N(N-1)(N-2)(N-3)} \sum_{i\ne j} d(i,j,\pi(i),\pi(j))+\frac{(N_1-1)(N-N_1-1)}{N(N-1)(N-2)(N-3)} \sum_{i\ne j} d(i,j,\pi(j),\pi(i))\\
&\quad-\left(\frac{N-N_1-1}{N(N-1)(N-2)}-\frac{2(N_1-1)(N-N_1-1)}{N(N-1)(N-2)(N-3)}\right)
\sum_{i\ne j}d(i,j,\pi(i),\pi(i))\\
&\quad-
\left(\frac{N-N_1-1}{N(N-1)(N-2)}-\frac{2(N_1-1)(N-N_1-1)}{N(N-1)(N-2)(N-3)}\right)
\sum_{i\ne j}d(i,j,\pi(j),\pi(j)) \\
&\quad +\frac{(N_1-1)(N-N_1-1)}{N(N-1)(N-2)(N-3)} \sum_{i\ne j}\sum_{k=1}^N d(i,j,k,k)\\
&=\frac{(N_1-1)(N-N_1-1)}{N(N-1)(N-2)(N-3)} \sum_{i\ne j} d(i,j,\pi(i),\pi(j))+\frac{(N_1-1)(N-N_1-1)}{N(N-1)(N-2)(N-3)} \sum_{i\ne j} d(i,j,\pi(j),\pi(i))\\
&\quad + \left(\frac1{N(N-1)}-\frac{4(N_1-1)(N-N_1-1)}{N(N-1)(N-2)(N-3)}\right) \sum_{i=1}^N d(i,i,\pi(i),\pi(i))\\
&\quad-\frac{(N_1-1)(N-N_1-1)}{N(N-1)(N-2)(N-3)} \sum_{i,k\in[N]} d(i,i,k,k)\\
&=\frac{(N_1-1)(N-N_1-1)}{N(N-1)(N-2)(N-3)} \sum_{i\ne j} d(i,j,\pi(i),\pi(j))+\frac{(N_1-1)(N-N_1-1)}{N(N-1)(N-2)(N-3)} \sum_{i\ne j} d(i,j,\pi(j),\pi(i)).
\end{align*}
Therefore, we have
$$\begin{aligned}
        &\quad\Er\left[\left.\sum_{i\in I,j\in I^c}\tilde{d}_{I,\pi(I)}(i,j,\pi(i),\pi(j))+\frac{(N_1-1)(N-N_1-1)}{N(N-1)(N-2)(N-3)}\sum_{i\ne j} d(i,j,\pi(j),\pi(i))\right|\pi\right]\\
        &=\Er\left[\left.\sum_{i\in I,j\in I^c}d(i,j,\pi(i),\pi(j))\right|\pi\right]-\Er\left[\left.\sum_{i\in I,j\in I^c}\frac{1}{|I|}\sum_{k\in \pi(I)} d(i,j,k,\pi(j))\right|\pi\right]\\
        &\quad-\Er\left[\left.\sum_{i\in I,j\in I^c}\frac{1}{|I^{c}|}\sum_{\ell\in \pi(I^{c})} d(i,j,\pi(i),\ell)\right|\pi\right]+\Er\left[\left.\sum_{i\in I,j\in I^c}\frac{1}{|I||I^c|}\sum_{k\in \pi(I),\ell\in\pi(I^c)} d(i,j,k,\ell)\right|\pi\right]\\
        &+\frac{(N_1-1)(N-N_1-1)}{N(N-1)(N-2)(N-3)}\sum_{i\ne j} d(i,j,\pi(j),\pi(i))\\
        &=\left(\frac{N_1(N-N_1)}{N(N-1)}-\frac{(N-N_1)(N-N_1-1)}{N(N-1)(N-2)}\right.\\
        &\left.\qquad-\frac{N_1(N_1-1)}{N(N-1)(N-2)}+\frac{(N_1-1)(N-N_1-1)}{N(N-1)(N-2)(N-3)}\right)\sum_{i\ne j} d(i,j,\pi(i),\pi(j))\\
        &\quad+\frac{(N_1-1)(N-N_1-1)}{N(N-1)(N-2)(N-3)}\sum_{i\ne j} d(i,j,\pi(j),\pi(i))\\
        &=\frac{(N_1-1)(N-N_1-1)\big(N^2 - 3N + 1\big)}{N(N-1)(N-2)(N-3)}\sum_{i\ne j} d(i,j,\pi(i),\pi(j)).
    \end{aligned}
$$
Thus for any convex function $\varphi$, we have
\begin{align*}
   &\quad\Er\left[\varphi\left(\sum_{i\ne j} d(i,j,\pi(i),\pi(j))\right)\right]\\
   &=\Er \left[\varphi\left(\Er\left[\left.\alpha\sum_{i\in I,j\in I^c}\tilde{d}_{I,\pi(I)}(i,j,\pi(i),\pi(j))+\beta\sum_{i\ne j} d(i,j,\pi(j),\pi(i))\right|\pi\right]\right)\right]\\
   &\leq\frac{1}{2}\Er \left[\varphi\left(2\alpha\sum_{i\in I,j\in I^c}\tilde{d}_{I,\pi(I)}(i,j,\pi(i),\pi(j))\right)\right]+ \frac{1}{2}\Er \left[\varphi\left(2\beta\sum_{i\ne j} d(i,j,\pi(j),\pi(i))\right)\right]\\
   &=\frac{1}{2}\Er \left[\varphi\left(2\alpha\sum_{i\in I,j\in I^c}\tilde{d}_{I,J}(i,j,\pi_1(i),\pi_2(j))\right)\right]+ \frac{1}{2}\Er \left[\varphi\left(2\beta\sum_{i\ne j} d(i,j,\pi(j),\pi(i))\right)\right],
\end{align*}
where the final equality follows because $\sum_{i\in I,j\in I^c}\tilde{d}_{I,\pi(I)}(i,j,\pi(i),\pi(j))$ and   $\sum_{i\in I,j\in I^c}\tilde{d}_{I,J}(i,j,\pi_1(i),\pi_2(j))$ share the same distribution, and can be shown similarly to \eqref{eq:converttopi1pi2}.
\end{proof}

{\small 
\bibliographystyle{apalike}
\bibliography{AMS}
}
\end{document}